\documentclass[12pt]{article}
\usepackage{amsmath, amsfonts, amssymb, amsthm, graphicx}
\usepackage{tikz}
\usetikzlibrary {graphs}
\usetikzlibrary[graphs]
\usepackage{tikz-cd}
\usepackage{quiver}
\usepackage[dvipsnames]{xcolor}
\usepackage{nameref}
\usepackage{hyperref}
\usepackage{float}
\usepackage{enumitem}
\usepackage{caption}
\usepackage{subcaption}
\usepackage{cleveref}

\theoremstyle{definition}

\newtheorem*{definition*}{Definition}
\newtheorem{example}{Example}
\newtheorem{counterexample}{Counterexample}
\newtheorem*{example*}{Example}
\newtheorem{observation}{Observation}
\newtheorem*{observation*}{Observation}

\DeclareMathOperator{\stab}{Stab}

\theoremstyle{plain}
\newtheorem{lemma}{Lemma}
\newtheorem*{lemma*}{Lemma}
\newtheorem{proposition}{Proposition}
\newtheorem*{proposition*}{Proposition}
\newtheorem{corollary}{Corollary}
\newtheorem*{corollary*}{Corollary}
\newtheorem{theorem}{Theorem}
\newtheorem*{theorem*}{Theorem}

\newcommand{\on}{\operatorname}
\newcommand{\Fer}{\operatorname{Fer}}
\newcommand{\fer}{{\rm Fer}}

\newcommand{\Aut}{\on{Aut}}
\newcommand{\Sym}{\on{Sym}}
\newcommand{\id}{\on{id}}
\newcommand{\E}{\mathcal{E}}

\title{Stem-Symmetry, Comb Products, and their Relation to Amoeba Graphs}

\author{
	Jillian Eddy$^\dagger$ 
	\and Ryan Pesak$^\dagger$ 
        \and Daniel Qin$^\dagger$ 
        \and Denae Ventura$^\ddagger \footnote{Corresponding author. Email dventuraarre@mtholyoke.edu.}$ 
	\\ \\ \\
	$^\dagger$ Dept. of Mathematics, University of California at Davis
\\
    $^\ddagger$ Dept. of Mathematics and Statistics, Mount Holyoke College}

\begin{document}

\maketitle

\begin{abstract}
Local and global amoebas are families of labeled graphs that satisfy interpolation properties on a fixed vertex set. A labeled graph $G$ on $n$ vertices is a \emph{local amoeba (resp. global amoeba)} if there exists a sequence of feasible edge-replacements between any two labelled embeddings of $G$ into $K_n$ (resp. $K_{n+1}$). Here, a feasible edge-replacement removes an edge and reinserts it so that the resulting graph is isomorphic to $G$; the induced relabeling yields a class of permutations of the label set. Motivated by classical group theoretic ideas, we introduce the \emph{hang group}, a new invariant that can encode how local amoebas embed into larger ones. Using this framework, we identify necessary and sufficient conditions connecting stem-symmetric and hang-symmetric graphs with local and global amoebas. In particular, we show how hang-symmetry and stem-symmetry conditions propagate under the addition of leaves and isolated vertices, in turn yielding constructive criteria for both local and global amoebas. Finally, via wreath products, we provide four sets of sufficient conditions, one for each property, guaranteeing when the comb product is a local amoeba, a global amoeba, stem-symmetric, or hang-symmetric. These results strengthen and generalize existing constructions of local and global amoebas.
\end{abstract}

\section{Introduction}

Amoeba graphs were first introduced in \cite{caro2021unavoidable} as examples of balanceable graphs, with a purely combinatorial definition. Later, \cite{caro2023graphs} provided an equivalent definition in group-theoretical terms by leveraging the notion of feasible edge-replacements and classifying amoebas into local and global types based on their interpolation properties. This group-theoretical perspective is the focus of our work.

Throughout, we work with graphs whose vertices are labeled. An edge-replacement of a labeled graph $G$ is a local operation that removes an edge $e \in E(G)$ and replaces it with another edge $e' \in E(\overline{G}) \cup \{e\}$. Such a replacement is feasible if the resulting graph is isomorphic to $G$ (respecting labels). A graph $G$ on $n$ vertices is a local amoeba if every labeled copy of $G$ in $K_n$ can be reached from $G$ by a sequence of feasible edge-replacements. A graph $G$ is a global amoeba if $G \cup K_1$ is a local amoeba.

Local and global amoebas have been investigated through various lenses. In \cite{caro2021unavoidable}, global amoebas were shown to be balanceable, and bipartite global amoebas to be omnitonal, a strengthening of balanceability. The group-theoretical study in \cite{caro2023graphs} provides concrete examples of both local and global amoebas, including a recursive family $\mathcal{T}$ of global amoebas known as Fibonacci-type trees. These and other families were later shown to be local amoebas via a recursive construction developed in \cite{eslava2023new}.

We are particularly inspired by the recursive construction of global amoebas in \cite{hansberg2021recursive}, which serves as a starting point for several of our results. Building on their ideas, we investigate how the comb product interacts with the groups governing  amoeba and adjacent properties. In the following subsection, we review the necessary group-theoretic definitions and preliminaries related to local and global amoebas.

\subsection{Preliminaries}
For any finite set $S$, let $\Sym(S)$ be the set of bijections $S \to S$; equivalently, the set of permutations on $S$. Recall that $\Sym(S)$ is a group isomorphic to $S_{|S|}$, the symmetric group on $|S|$ elements. In this work, every graph $G$ is equipped with a bijective labeling $\lambda:V(G)\to X$ on their vertex set such that $v_x = \lambda ^{-1}(x)$ for each $x\in X$. To avoid notational clutter, we at times use the vertex and its corresponding label interchangeably when the labeling is clear from context. Let $L_G=\{ij \mid v_iv_j\in E(G)\}$ be the set of edge labels of $E(G)$ where there is no distinction between $ij$ and $ji$. For every $\sigma\in \Sym(X)$, we define an embedding $G_\sigma$ which has vertex set $V(G)$ and edge set $$E(G_\sigma)= \{v_iv_j \mid \sigma(i)\sigma(j)\in L_G \}.$$ Note that $G_\sigma\cong G$ for all $\sigma\in \Sym(X).$ 

If $G_1$ and $G_2$ are graphs, we say that $G_1=G_2$ if $V(G_1)=V(G_2)$ and $E(G_1)=E(G_2)$ (if they are labeled graphs, their labelings are also the same). Note that given an unlabeled copy $G'$ of $G$, there are $|\Aut(G)|$ different ways of labeling $G'$ so that the labels correspond to those on $G$. In particular, this implies that the set $A_G=\{\sigma \in S_n \mid G_\sigma = G' \}$ has $|\Aut(G)|$ elements, and furthermore, $A_G\cong \Aut(G)$.

The use of labels on the vertices is important to keep track of the role each vertex and edge has in copies of $G$. In $G_\sigma$, the vertex labeled $i$ corresponds to the copy of vertex $v_i$ of $G$, while the edge labeled $ij$ represents the copy of the edge $v_i v_j \in E(G)$. In particular, $L_{G_{\sigma}}= L_G$ for all $\sigma \in \Sym(X)$. See \Cref{fig:example-comparison} for an example.

\begin{figure}[H]
    \centering
    \begin{subfigure}[b]{0.48\textwidth}
        \centering
        \begin{tikzpicture}
            \node[shape=circle,draw=black,inner sep=2pt,font=\small] (3) at (0,0) {$v_3$};
            \node[shape=circle,draw=black,inner sep=2pt,font=\small] (2) at (1.299,-0.75) {$v_2$};
            \node[shape=circle,draw=black,inner sep=2pt,font=\small] (1) at (1.299,0.75) {$v_1$};
            \node[shape=circle,draw=black,inner sep=2pt,font=\small] (4) at (-1.5,0) {$v_4$};
            \node[shape=circle,draw=black,inner sep=2pt,font=\small] (5) at (-3,0) {$v_5$};

            \node[blue,font=\small] (a) at (0,.7) {$3$};
            \node[blue,font=\small] (b) at (1.299,-.1) {$2$};
            \node[blue,font=\small] (c) at (1.299,1.4) {$1$};
            \node[blue,font=\small] (d) at (-1.5,.7) {$4$};
            \node[blue,font=\small] (e) at (-3,.7) {$5$};

            \path [] (1) edge (3);
            \path [] (3) edge (2);
            \path [] (3) edge (4);
            \path [] (4) edge (5);
        \end{tikzpicture}
        \caption{$G = G_{id}$ with $\lambda(v_i) = i$.}
        \label{fig1:example}
    \end{subfigure}
    \hfill
    \begin{subfigure}[b]{0.48\textwidth}
        \centering
        \begin{tikzpicture}
            \node[shape=circle,draw=black,inner sep=2pt,font=\small] (3) at (0,0) {$v_3$};
            \node[shape=circle,draw=black,inner sep=2pt,font=\small] (2) at (1.299,-0.75) {$v_2$};
            \node[shape=circle,draw=black,inner sep=2pt,font=\small] (1) at (1.299,0.75) {$v_1$};
            \node[shape=circle,draw=black,inner sep=2pt,font=\small] (4) at (-1.5,0) {$v_4$};
            \node[shape=circle,draw=black,inner sep=2pt,font=\small] (5) at (-3,0) {$v_5$};

            \node[blue,font=\small] (a) at (0,.7) {$4$};
            \node[blue,font=\small] (b) at (1.299,-.1) {$2$};
            \node[blue,font=\small] (c) at (1.299,1.4) {$1$};
            \node[blue,font=\small] (d) at (-1.5,.7) {$5$};
            \node[blue,font=\small] (e) at (-3,.7) {$3$};

            \draw[] (5) to [bend right=30] (3);
            \draw[] (5) to [bend left=30] (1);
            \draw[] (5) to [bend right=30] (2);
            \draw[] (3) to [bend right=20] (4);
        \end{tikzpicture}
        \caption{$G_{(345)}$ with $\sigma = (345)$.}
        \label{fig2:example}
    \end{subfigure}
    \caption{Graphs $G$ and $G_{(345)}$ with labels in blue and $L_G = \{13,23,34,45\}$. Notice that $L_G = L_{G_{(345)}}$.}
    \label{fig:example-comparison}
\end{figure}
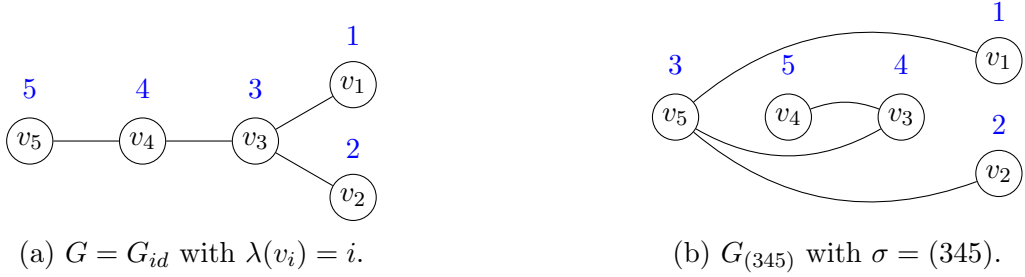

We now define edge-replacements in terms of elements in $L_G$. Given $G$ a labeled simple graph, we denote by $rs\to kl$ (or $e\to e'$ if $rs$ and $kl$ are the labels of edges $e$ and $e'$, respectively) the edge-replacement corresponding to the removal of the edge $e\in E(G)$ with labels $rs$ and the addition of the edge $e'\in E(\overline{G})\cup e$ with labels $kl$. An edge-replacement is said to be \emph{feasible} if the resulting graph $G-v_r v_s + v_kv_l$ is isomorphic to $G$. We denote by $\emptyset\to\emptyset$ the \emph{neutral edge-replacement}, where no edge is replaced. Let the set $$R_G = \{ rs\to kl \mid G-v_r v_s + v_kv_l \cong G, rs\neq kl \} \cup \{ \emptyset\to\emptyset \}$$

\noindent be the set of all feasible edge-replacements of $G$ given by their labels and let $R^* = R_G \setminus \{\emptyset \to \emptyset\}$. Now, for $e \to e' \in R^{*}_G$, define
\[
    \Fer _G (e\to e') = \{\sigma \in \Sym(X) \mid G - e_1 + e_2 = G_\sigma\}. 
\]
Note that in regard to the neutral edge-replacement, $\Fer _G (\emptyset\to \emptyset)=A_G$. This implies that if $e \to e'$ is feasible then $\Fer_G(e \to e')$ is nonempty. Importantly, the elements of $\Fer_G(e\to e')$ represent the $|\Aut(G)|$ different copies of $G$ that can be obtained via the feasible edge-replacement $e\to e' \in R_G$, which implies the following observation.

\begin{observation}~\label{lem:cosets}
Let $G$ be a graph with a label set $X$ and labeling $\lambda: V(G) \to X$. If $e_1 \to e_2$ is a feasible edge-replacement, then $\Fer_G(e_1 \to e_2)$ is a left coset of $\Aut(G)$. 
\end{observation}

\begin{proof}
Let $\sigma \in \Fer_G(e_1 \to e_2)$. We wish to show that $\Fer_G(e_1 \to e_2) = \sigma\Aut(G)$.\\ Let $\tau \in \Aut(G)$. Then 
\[
    G_{\sigma\tau} = (G_\tau)_\sigma = G_\sigma = G - e_1 + e_2.
\]
Thus, $\sigma\tau \in \Fer_G(e_1 \to e_2)$ for any $\tau \in \Aut(G)$. It follows that $\sigma\Aut(G) \subseteq \Fer_G(e_1 \to e_2)$.

Now let $\tau \in \Fer_G(e_1 \to e_2)$. Since $G_\sigma = G - e_1 + e_2$, then $(G - e_1 + e_2)_{\sigma^{-1}} = G$. Then
\[
    G_{\sigma^{-1}\tau} = (G_\tau)_{\sigma^{-1}} = (G - e_1 + e_2)_{\sigma^{-1}} = G.
\]
Therefore $\sigma^{-1}\tau \in \Aut(G)$. Thus $\tau = \sigma(\sigma^{-1}\tau) \in \sigma\Aut(G)$ and $\Fer_G(e_1 \to e_2) \subseteq \sigma\Aut(G)$. 
\end{proof}

The authors of \cite{caro2023graphs} observed that feasible edge-replacements had meaningful group structure given by 
permuting labels. In particular, it led to the definition of the \textit{feasible edge-replacement group} $\fer(G)$ of a graph $G$, or \textit{Fer group} for short,  generated by the set $\mathcal{E}_G= \bigcup_{e\to e' \in R_G}\fer_G (e\to e')$ which contains the permutations associated to feasible edge-replacements in $G$. Hence, $$\fer(G) = \langle \mathcal{E}_G \rangle.$$ By \Cref{lem:cosets}, we see that  $\mathcal{E}_G$ is partitioned into left cosets $\fer_G (e\to e')$ and is therefore closed under left-multiplication by $\Aut(G)$.

    Let $G$ be a graph with a labeling $\lambda: V(G) \to X$. We say that $G$ is a \emph{local amoeba} if $\fer(G)=\Sym(X)$. We say that $G$ is a \emph{global amoeba} if there is an integer $t\geq 1$ such that $G\cup t K_1$ is a local amoeba. Consider the set $\mathcal{E}_G^{i}$ of all permutations associated to edge-replacements in $R_G$ that fix the label $i \in X$, i.e.,
$$\mathcal{E}_G^{i}= \mathcal{E}_G \cap \stab_{\fer(G)}(i).$$ 

Let $\fer^{i}(G)$ be the subgroup of $\stab_{\fer(G)}(i)$ generated by the set $\mathcal{E}_G^{i}$. These subgroups restrict feasible edge-replacements that can be inherited to a larger graph as seen through the following lemma.

\begin{lemma}[\cite{eslava2023new}]\label{lemma:extends}
Let $H$ and $J$ be two vertex disjoint graphs provided with their corresponding disjoint sets of labels $X$ and $Y$. Consider vertices $v_x \in V(H)$, $v_y \in V(J)$ with labels $x \in X$ and $y \in Y$, respectively, and the graph $G = (H \cup J)+v_xv_y$ with the inherited set of labels $X \cup Y$. If $\alpha \in \mathcal{E}^{x}_H$, then $\alpha \cup {\rm id}_{\fer(J)} \in \mathcal{E}^{x}_G$.
\end{lemma}

Motivated by the group characterization of local amoebas, the authors of \cite{eslava2023new} defined an analogous class of graphs $G$ with a fixed label $i$ for which $\fer^{i}(G)$ is isomorphic to a symmetric group on $n(G)-1$ elements.

Let $G$ be a graph and $v \in V(G)$, and let $\lambda: V(G) \to X$ be a labeling of $G$ and $b = \lambda(v)$. We say that $G$ is \emph{stem-symmetric at $v$} if $\fer^{b} (G)\cong S_{n(G)-1}.$ When the labeling is clear from context, by abuse of notation, we may also say that $G$ is stem symmetric at the label $b$. 

Stem-symmetric graphs $G$ are useful for constructing and detecting local amoebas. The following lemma states that a stem-symmetric graph $G$ is, in fact, a local amoeba provided we can exhibit a suitable permutation.

\begin{lemma}[\cite{eslava2023new}]\label{lem:stem-sym_local}
    Let $G$ be a labeled graph that is stem-symmetric with respect to a vertex $v$, whose label is $b$. If
    \[\fer(G) \setminus {\rm Stab}_G(b) \neq \emptyset,\] 
    then $G$ is a local amoeba.
\end{lemma}

Throughout this work, we will use the fact that $G$ is a \emph{global amoeba} if and only if $G\cup K_1$ is a local amoeba \cite{caro2023graphs}. This result implies the following.

\begin{lemma}[\cite{caro2021unavoidable}]\label{lem:local-am with leaf implies local and global}
    If $G$ is a local amoeba with $\delta(G) \in \{0, 1\}$, then $G \cup K_1$ is a local amoeba, and so $G$ is a global amoeba.

\end{lemma}

\Cref{thm:recursion} states a general method by which local amoebas can be constructed via a recursion. Global amoebas can be constructed as well if the minimum degree is at most $1$. Authors in \cite{eslava2023new} have given recursive constructions for various families of local amoebas. 

\begin{theorem}[\cite{eslava2023new}]\label{thm:recursion}
    Let $H_1, J_1, H_2, J_2$ be vertex disjoint graphs provided with the roots $u, v, w, y$, respectively, and such that $H_1 \cong H_2$ and $u$ is similar to $w$. Let $H = (H_1 \cup J_1) + uv$ be stem-symmetric at $v$ and let $J= (H_2 \cup J_2) + wy$ be stem-symmetric with respect to $w$. Let $G = (H \cup J) + vw$ be labeled and let $b$ the label on $v$. Then we have the following facts.
    \begin{enumerate}
        \item[(i)] $G$ is stem-symmetric at $v$.
        \item[(ii)] If $\fer(G) \setminus {\rm Stab}_G(b) \neq \emptyset$, then $G$ is a local amoeba.
        \item[(iii)] If $\fer(G) \setminus {\rm Stab}_G(b) \neq \emptyset$ and $\delta(G) \le 1$, then $G$ is a global and a local amoeba.
    \end{enumerate}
    
\end{theorem}

In \cite{hansberg2021recursive}, the authors introduced a way of gluing global amoebas which resulted in larger global amoebas. They defined a graph $G$ with label set $X$ and a root labeled $k$ to be stem-transitive if there is a set  $S\subseteq \stab_{\fer(G)}(k)$ such that $\langle S \rangle$ acts transitively on $X\setminus \{k\}$. The importance of this definition lies in the fact that the vertex with label $k$ must not move for any reason. However if we use the previous definition, there may be permutations that belong to $S\subseteq \stab_{\fer(G)}(k)$ which may correspond to a series of \textit{multiple} feasible edge-replacements where some them may move $k$, but their composition does not. In direct communication with one of the authors, we agreed that the definition of a stem-transitive graph should be as follows. Let $G$ be a graph with label set $X$ and a root labeled $i$. We say that $G$ is \emph{stem-transitive} at its root if $\langle \E_G^i \rangle$ acts transitively on $X \setminus \{i\}$. Using $\E_G^i$ instead of $\stab_{\Fer(G)}(i)$ ensures that for a permutation which is a composition of \textit{multiple} feasible edge replacements, each component of the composition fixes the label $i$.

If there is a vertex $v_j$, with $j\neq k$, such that there is a permutation $\varphi\in A_G$ with $\varphi(k)=j$, then $v_j$ is called a \emph{root-similar vertex}. Using a characterization of global amoebas \cite{caro2023graphs}, the authors proved that a rooted graph $G$ with $\delta(G)=1$ which is stem-transitive and has a root-similar vertex is a global amoeba. A graph with these properties is called a \emph{double-rooted global amoeba}. Moreover, a local amoeba with these properties is called \emph{double-rooted local amoeba}. The authors provided three different recursive constructions of local and global amoebas using these properties.

\bigskip
\noindent \textbf{Notation.} Throughout this work, our results make use of the following graph constructions. If $G$ is a graph rooted at a vertex with label $i$, let $G^*$ be the rooted graph $G$ with an isolated vertex with the root inherited from $G$. Let $G^{\dagger}$ be the graph $G$ with a leaf connected to the root of $G$. In this case, $G^\dagger$ is rooted at the newly-added leaf.

\bigskip

\subsection{Our Contributions and Main Results}

The remaining sections of this paper present our contributions to the theory of amoeba graphs.  We develop new group-theoretic invariants and constructive methods for local and global amoebas through the lenses of stem- and hang-symmetry.

In \Cref{sec: The Hang Group}, we define the \emph{hang group} of a graph $G$ at a vertex $v \in V(G)$ and \emph{hang-symmetry} as an analog to stem-symmetry as introduced in \cite{eslava2023new}.  We show that for a graph G and the graph $G^\dagger$ formed by adding a leaf to the root of $G$, hang-symmetry and stem-symmetry are equivalent properties:

\bigskip

\noindent\textbf{\Cref{prop:hang-group}.}
\emph{Let $G$ be a rooted graph with labeling $\lambda : V(G) \to Y$ such that $j$ is the label of its root $G$. Consider the graph $G^\dagger$, where $i$ is the label of the newly added leaf. Then 
\begin{center}
 $G^\dagger$ is stem-symmetric at $v_i$ if and only if $G$ is hang-symmetric at $v_j$.
\end{center}}

In 
\Cref{sec: stem hang-sym global}, we show equivalent conditions for stem-symmetry at a vertex within $G^*$. 
\bigskip 

\noindent\textbf{\Cref{thm:stem-sym-global}.}\emph{
    Let $G$ be a labeled graph rooted at a vertex labeled $i$ and let $X$ be the label set of $G^*$ with the new isolated vertex labeled $j$. Then the following are equivalent:
    \begin{enumerate}[label = (\alph*)]
        \item $G^*$ is stem-symmetric at $v_i$;
        \item $\Fer^i(G^*)$ acts transitively on $X \setminus \{i\}$.
        \item Every orbit of $\Fer^i(G)$ acting on $X \setminus \{i, j\}$ contains a leaf.
    \end{enumerate}}
    \bigskip 
    
This is then leveraged to establish a new characterization of global amoebas thereby extending the existing list found in~\cite[Theorem 15 (iii)]{caro2023graphs}:
\bigskip 

\noindent\textbf{\Cref{cor: global amoeba = transitive}.}
\emph{Let $G$ be a graph. Then $G$ is a global amoeba if and only if $\Fer(G^*)$ acts transitively on its set of labels.}

\bigskip 
An immediate consequence of this characterization is that $\Fer$ groups of graphs with an isolated vertex are generically intransitive; Not all transitive groups are symmetric, yet Fer groups of such graphs are either intransitive or symmetric. In \Cref{sec: Comb Product of Amoebas}, we pivot to interactions between the local amoeba property and comb products of graphs. In particular, we provide a set of sufficient conditions for when the comb product of a pair of graphs is a local amoeba: 
\bigskip 

\noindent\textbf{\Cref{cor:big_corollary}.}
\emph{Let $G$ be a local amoeba on $m$ vertices with a leaf and let $H$ be a rooted graph on $n$ vertices which is hang-symmetric at its root. Then $\Fer(G * H)$ is either equal to $S_m \wr S_n$ or $S_{mn}$, and in the latter case, $G * H$ is a local amoeba.
}\bigskip 

This is guaranteed when $H$ is a disconnected local amoeba (see~\Cref{lem:disconnected_comb_product}) or when $H$ is a path (see~\Cref{thm:path-amoeba}). An analogous set of sufficient conditions is obtained for when the comb product is a global amoeba:
\bigskip

\noindent\textbf{\Cref{thm-ifH*hs-G*Hglobam}.}
\emph{Let $G$ be a global amoeba and let $H$ be a rooted graph. If $H^*$ is hang-symmetric at its root, then $G * H$ is a global amoeba.}
\bigskip

Note that the comb product was previously studied in \cite[Theorem 3.10]{hansberg2021recursive}, and the sufficient conditions above strictly extend their result. In \Cref{sec: Stem-Symmetry Hang-Symmetry under the Comb product}, we examine when stem-symmetry and hang-symmetry are preserved under the comb product. In particular, if $G^*$ and $H^*$ are stem- (resp. hang-) symmetric, we find that $(G * H)^*$ is as well:
\bigskip

\noindent\textbf{\Cref{thm:when_is_comb_product_stem_sym_global}.}
\emph{If $G^*$ is stem-symmetric at its root labeled $i$ and $H^*$ is stem-symmetric at its root labeled $j$ with a root-similar vertex, then $(G * H)^*$ is stem-symmetric at the label $(j, i)$.}

\bigskip
\noindent\textbf{\Cref{cor:(G*H)* hang sym}.}
\emph{If $G^*$ is hang-symmetric with root labeled $i$ and $H^*$ is hang-symmetric with root labeled $j$, then $(G * H)^*$ is hang-symmetric with root labeled $(j, i)$.}
\bigskip

In some cases these properties can be propagated via iterated comb products, e.g., iterated comb products of paths. Analogous statements to the main theorem in the previous section also appear in this section. For example:
\bigskip

\noindent\textbf{\Cref{cor:hangsym under comb}.}
\emph{If $G$ has a leaf and is hang-symmetric at a vertex labeled $i$ and $H$ is hang-symmetric at a vertex labeled $j$, then $\bold{H}_{(j, i)}(G * H)$ is either equal to $S_m \wr S_n$ or $S_{mn}$.}
\bigskip

   \Cref{sec: Conclusion and open problems} is our conclusion which contains a summary and list of open problems for future research. To illustrate the relationships between the local/global amoeba properties and other properties discussed in this paper, we include two flow charts in which arrows represent implications:

\begin{figure}[H]
    \centering
\[\begin{tikzcd}
	& \boxed{\begin{array}{c} \substack{G \text{ is stem-symmetric at its} \\ \text{root, has a leaf, and} \\ \text{has a root-similar vertex}} \end{array}}\\
	\boxed{\begin{array}{c} \substack{G \text{ is hang-symmetric at its} \\ \text{root, and has a leaf}} \end{array}} & \boxed{\begin{array}{c} \substack{G^* \text{ is stem-symmetric at its root, has} \\ \text{a root-similar vertex, and has an} \\ \text{ isolated vertex which is not its root}} \end{array}} & \boxed{\begin{array}{c} \substack{G \text{ is stem-symmetric} \\ \text{at its root, and has} \\ \text{a leaf that is not its root}} \end{array}}\\
	\boxed{\begin{array}{c} \substack{G^{\dagger*} \text{ is stem-symmetric at the} \\  \text{newly-added leaf, and has} \\ \text{a leaf which is not the root}} \end{array}} & \boxed{\begin{array}{c} \substack{G^* \text{ is hang-symmetric at its} \\ \text{root, and has an isolated} \\ \text{vertex which is not the root}} \end{array}} & \boxed{\begin{array}{c} \substack{G^* \text{ is stem-symmetric at} \\ \text{a non-isolated vertex}} \end{array}} \\
	& \boxed{\begin{array}{c} \substack{G^{\dagger*} \text{is stem-symmetric at} \\ \text{the newly-added leaf}} \end{array}} &\boxed{{G \text{ is a global amoeba}}}
	\arrow[from=1-2, to=2-1, Rightarrow]
	\arrow[from=1-2, to=2-2, Rightarrow]
	\arrow[from=1-2, to=2-3, Rightarrow]
	\arrow[tail reversed, from=2-1, to=3-1, Leftrightarrow]
	\arrow[from=2-1, to=3-2, Rightarrow]
	\arrow[from=2-2, to=3-2, Rightarrow]
	\arrow[from=2-2, to=3-3, Rightarrow]
	\arrow[from=2-3, to=3-3, Rightarrow]
	\arrow[from=3-1, to=4-2, Rightarrow]
	\arrow[tail reversed, from=3-2, to=4-2, Leftrightarrow]
	\arrow[from=3-2, to=4-3, Rightarrow]
\end{tikzcd}\]
    \caption{Diagram of implications related to the global amoeba property}
    \label{fig:placeholder}
\end{figure}
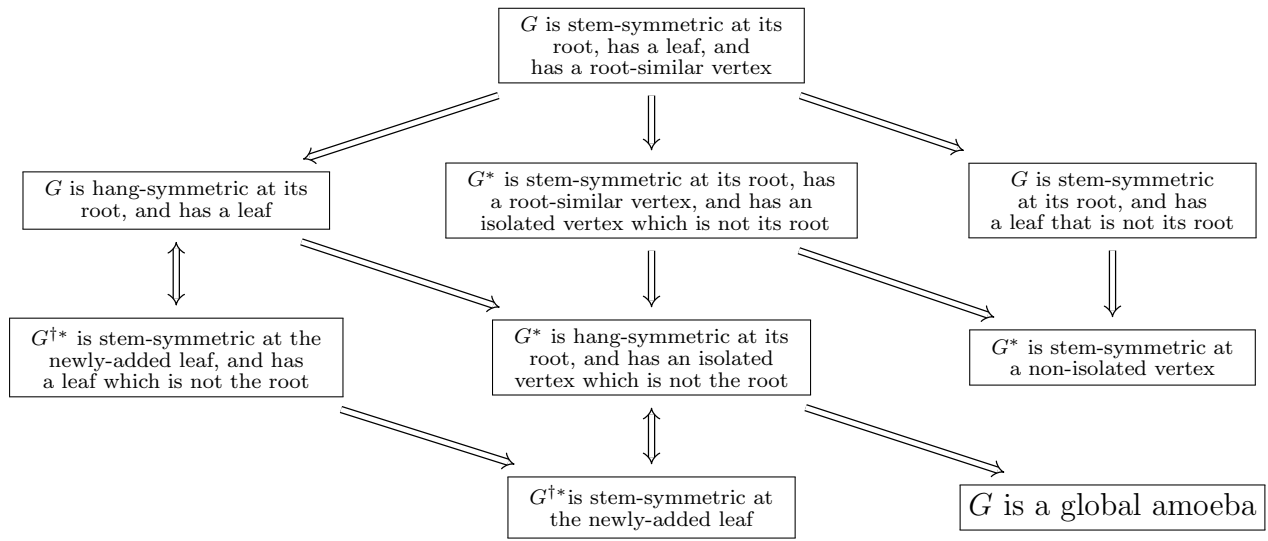

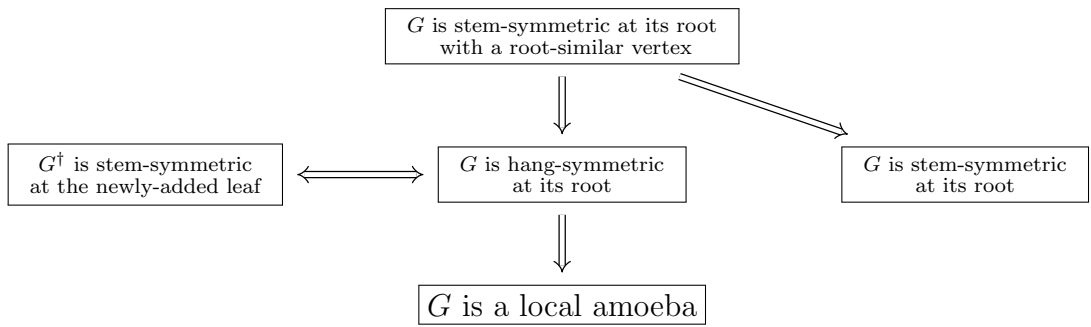
\begin{figure}[H]
    \centering
    \[\begin{tikzcd}
	& \boxed{\begin{array}{c} \substack{G \text{ is stem-symmetric at its root} \\ \text{ with a root-similar vertex}} \end{array}} \\
	\boxed{\begin{array}{c} \substack{G^\dagger \text{ is stem-symmetric} \\ \text{at the newly-added leaf}} \end{array}} & \boxed{\begin{array}{c} \substack{G \text{ is hang-symmetric} \\ \text{at its root}} \end{array}} & \boxed{\begin{array}{c} \substack{G \text{ is stem-symmetric} \\ \text{at its root}} \end{array}}\\
	& \boxed{{G \text{ is a local amoeba}}}
	\arrow[from=1-2, to=2-2, Rightarrow]
	\arrow[from=1-2, to=2-3, Rightarrow]
	\arrow[tail reversed, from=2-1, to=2-2, Leftrightarrow]
	\arrow[from=2-2, to=3-2, Rightarrow]
\end{tikzcd}\]

    \caption{Diagram of implications related to the local amoeba property}
    \label{fig:placeholder}
\end{figure}

\section{The Hang Group}~\label{sec: The Hang Group}
In this section, we define the \emph{hang group} of a graph $G$ at a vertex $v \in V(G)$, and discuss its relation to stem-symmetry at a leaf vertex. Let $G$ be a rooted graph with labeling $\lambda: V(G)\to X$ and let $v_i \in V(G)$ be the root such that $\lambda(v_i)=i$. The \emph{hang group} of $G$ at $i$ is the group $\mathbf{H}_i(G) := \langle \mathcal{E}^i_G \cup \Aut(G) \rangle$. We say $G$ is \emph{hang-symmetric with respect to $v_i$}, or simply \emph{hang-symmetric} (when the root is clear from context), if $\mathbf{H}_i(G) = \Sym(X)$. When the labeling is clear from context, by abuse of notation, we may also say that $G$ is hang-symmetric at the label $i$. 

Note that if $G$ is hang-symmetric then it is also a local amoeba. The notions of double-rootedness and hang-symmetry are related, but they are not equivalent. If $G$ is a rooted, stem-symmetric graph with a root-similar vertex, then it is also hang-symmetric by \Cref{lem:stem-sym_local}. However, the reverse implication is false. See \Cref{fig:hang-symm}.

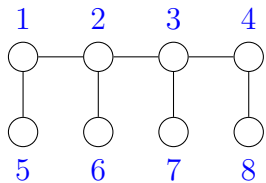
\begin{figure}[h]
  
\begin{center}
\begin{tikzpicture}
   
    \node[shape=circle,draw=black] (1) at (-1.5,0) {};
    \node[shape=circle,draw=black] (2) at (-.5,0) {};
    \node[shape=circle,draw=black] (3) at (.5,0) {};
    \node[shape=circle,draw=black] (4) at (1.5,0) {};
    \node[shape=circle,draw=black] (5) at (-1.5,-1) {};
    \node[shape=circle,draw=black] (6) at (-.5,-1) {};
    \node[shape=circle,draw=black] (7) at (.5,-1) {};
    \node[shape=circle,draw=black] (8) at (1.5,-1) {};
    \node[blue] (a) at (-1.5,.5) {$1$};
    \node[blue] (b) at (-.5,.5) {$2$};
    \node[blue] (c) at (.5,.5) {$3$};
    \node[blue] (d) at (1.5,.5) {$4$};
    \node[blue] (a) at (-1.5,-1.5) {$5$};
    \node[blue] (b) at (-.5,-1.5) {$6$};
    \node[blue] (c) at (.5,-1.5) {$7$};
    \node[blue] (d) at (1.5,-1.5) {$8$};
    \path [] (1) edge node[left] {} (2);
    \path [] (2) edge node[left] {} (3);
    \path [] (3) edge node[left] {} (4); 
    \path [] (1) edge node[left] {} (5);
    \path [] (2) edge node[left] {} (6);
    \path [] (3) edge node[left] {} (7);
    \path [] (4) edge node[left] {} (8);
   
\end{tikzpicture} 
\end{center}
\caption{A graph $G$ that is hang-symmetric with respect to $1$. The sets $\mathcal{E}^1_G = \{ (24)(68), (34)(78), (48), (47) \}$ and $Aut(G)=\{(14)(58)(23)(67)\}$ satisfy that  $\langle \mathcal{E}^1_G \cup Aut(G) \rangle = S_8 $, but $\langle \mathcal{E}^1_G \rangle \neq S_7$.}
\label{fig:hang-symm}
\end{figure}

There is an equivalent way to define $\mathbf{H}_i(G)$, which also gives a characterization of when a graph is stem-symmetric at a leaf. We employ \Cref{lemma:extends} and the following lemma to prove this.

\begin{lemma}\label{lem:injective-homomorphism}
Let $X \subseteq Y$ be sets such that $|Y| = |X| + 1$. If $Y \setminus X = \{y\}$, then the map $f: \Sym(X) \to \on{Stab}_{\Sym(Y)}(y)$ given by $f(\sigma) = \sigma \cup \id_{Y \setminus X}$ is an isomorphism of groups. Moreover, for any $x \in X$ and $\sigma \in \Sym(X)$, we have $\sigma(x) = f(\sigma)(x)$. 
\end{lemma}

That is, the action of $\Sym(X)$ on $X$ is essentially the same as the action of $\on{Stab}_{\Sym(Y)}(y)$ on $X$. By adjoining $\id_{Y \setminus X}$ to each permutation, we are not affecting $Y \setminus X$. While Lemma~\ref{lem:injective-homomorphism} follows immediately, it is important for giving a rigorous argument for the proof of \Cref{prop:hang-group}.

\begin{proposition}\label{prop:hang-group}
Let $G$ be a rooted graph with labeling $\lambda : V(G) \to Y$ such that $\lambda(v_x)=x$ for all $x \in Y$. Moreover, let $j$ be the label of the root of $G$, and then consider the graph $G^\dagger$, where $i$ is the label of the newly added leaf. Then 

\begin{enumerate}
    \item[i)] the map $f:\mathcal{E}_{G}^j \cup \Aut(G) \to \mathcal{E}_{G^\dagger}^i$ given by $\sigma \mapsto \sigma \cup \id_{\Fer(G^\dagger \setminus G)}$ is a bijection,
    \item[ii)] the map $\tilde{f}:\mathbf{H}_j(G) \to \Fer^i(G^\dagger)$ given by $ \sigma \mapsto \sigma \cup \id_{\Fer(G^\dagger \setminus G)}$ is an isomorphism, and
    \item[iii)] $G^\dagger$ is stem-symmetric at $v_i$ if and only if $G$ is hang-symmetric at $v_j$.
\end{enumerate}
\end{proposition}

\begin{proof}[ Proof of item $i)$]
Let $X = Y \setminus \{i\}$.  \Cref{lem:injective-homomorphism} states that $f$ maps $\Sym(X)$ injectively and surjectively onto $\on{Stab}_{\Sym(Y)}(i)$. This means that every $\tau \in \on{Stab}_{\Sym(Y)}(i)$ may be written $\tau = f(\sigma)$ for a \textit{unique} $\sigma \in \Sym(X)$. In the sequel, we will write $f(\sigma) \in \on{Stab}_{\Sym(Y)}(i)$ to denote an arbitrary element of $\on{Stab}_{\Sym(Y)}(i)$, as every such element may be written this way.

Let $S = \mathcal{E}_{G}^j \cup \Aut(G) \subset \Sym(X)$. Since every permutation in $\mathcal{E}_{G^\dagger}^i$ fixes the label $i$, then $\mathcal{E}_{G^\dagger}^i \subset \on{Stab}_{\Sym(Y)}(i)$. Our goal is to show that $f$ establishes a bijection between $S$ and $\mathcal{E}_{G^\dagger}^i$, however \Cref{lem:injective-homomorphism} states that $f$ is already bijective between the supersets $\Sym(X) \supset S$ and $\on{Stab}_{\Sym(Y)}(i) \supset \mathcal{E}_{G^\dagger}^i$. Thus, to establish that $f$ maps bijectively from $S$ to $\mathcal{E}_{G^\dagger}^i$, it suffices to show that $f(S) = \mathcal{E}_{G^\dagger}^i$, which is equivalent to proving that 

\begin{center}
  $\sigma \in S  $ if and only if $f(\sigma) \in \mathcal{E}_{G^\dagger}^i.$   
\end{center}

Suppose $\tau \in S$. Then either $\tau \in \mathcal{E}_{G}^j$ or $\tau \in \Aut(G)$. If $\tau \in \mathcal{E}_{G}^j$, then \Cref{lemma:extends} states that $f(\tau) = \tau \cup \id_{\Fer(G^\dagger \setminus G)} \in \mathcal{E}^j_{G^\dagger}$. Note that, $f(\tau) = \tau \cup \id_{\Fer(G^\dagger \setminus G)}$ must fix the label $i$. Thus, $f(\tau) \in \mathcal{E}_{G^\dagger}^i$.

Now, suppose $\tau \in \Aut(G)$. Note that since $v_i$ is a leaf, then there is a splitting $G^\dagger = G \cup U + ij$, where $U = G^\dagger \setminus G$ is the graph on the single vertex $v_i$. Then $G_{\tau} = G$ since $\tau$ is an automorphism of $G$. If we put $\tau(j) = k$, we may calculate
\begin{align*}
    G^\dagger_{f(\tau)} &= (G \cup U + ij)_{f(\tau)} = G_{\tau} \cup U_{\id_{\Fer(U)}} + f(\tau)(i)f(\tau)(j) = G \cup U + i\tau(j)\\
    &= G \cup U + ij - ij + i\tau(j) = G^\dagger - ij + ik.
\end{align*}

This implies that 
$f(\tau)=\tau \cup id_{\Fer(G^\dagger\setminus G)}$ is a permutation associated to the feasible edge-replacement $ij\to ik$ in $G^\dagger$, which means that $f(\tau) \in \mathcal{E}_{G^\dagger}^i$. Thus, in every case, we have the implication $\tau \in S \implies f(\tau) \in \mathcal{E}_{G^\dagger}^i$.\\

Now, suppose that $f(\mu) \in \mathcal{E}_{G^\dagger}^i$ is an arbitrary element of $\mathcal{E}_{G^\dagger}^i$. Then there exists an edge-replacement $ab \to cd$ such that $G_{f(\mu)} = G^\dagger - v_av_b + v_cv_d$. Since $v_i$ is a leaf of $G^\dagger$, then $E(G^\dagger) = E(G) \sqcup \{ij\}$. First, suppose that $ab = ij$. Then $cd = ik$ for some $k \in X$, otherwise we isolate the vertex $v_i$, and therefore move the label $i$, which contradicts $f(\mu) \in \mathcal{E}_{G^\dagger}^i$. The vertex which receives the label $i = f(\mu)(i)$ must be connected to the vertex which receives the label $f(\mu)(j)$. Since (once we perform the edge-replacement) the vertex $v_i$ is only connected to the vertex $v_k$, we must have $k = \lambda(v_k) = f(\mu)(j) = \mu(j)$. Now, we compute
\begin{align*}
    G_\mu \cup U + ik &= G_\mu \cup U_{\id_{\Fer(G^\dagger \setminus G)}} + f(\mu)(i)f(\mu)(j) = (G \cup U + ij)_{f(\mu)}\\
    &= G^\dagger_{f(\mu)} = G^\dagger - ij + ik = (G \cup U + ij) - ij + ik = G \cup U + ik.
\end{align*}
When we restrict our vertices to the set $V(G)$ on each side of this equation, we get $G_\mu = G$. Thus, $\mu \in \Aut(G) \subseteq S$, and so $f(\mu) \in \mathcal{E}_{G^\dagger}^i $ implies that $ \mu \in S$ if $ab = ij$.

For the second case, suppose that $ab \neq ij$. Since $E(G^\dagger) = E(G) \sqcup \{ij\}$, then $ab \in E(G)$. Moreover, we claim $c, d \in V(G)$. Otherwise, we must have $c = i$ without loss of generality. Then once we perform the edge-replacement $ab \to id$, the label $i$ is not moved, so we have added an edge to the vertex $v_i$ without taking one away. Thus, $v_i$ is not a leaf anymore, so it must move, contradicting $f(\mu) \in \mathcal{E}_{G^\dagger}^i$. It follows that $c, d \neq i$, and so $c, d \in V(G)$. 

Since $ab \in E(G)$ and $c, d \in V(G)$, then we may consider $ab \to cd$ as an edge-replacement on $G$ instead of an edge-replacement on $G^\dagger$. Moreover, when we perform the edge-replacement $ab \to cd$ and relabel, the label $i$ is fixed, and the edge $ij$ is not moved, so $v_j$ must receive the label $j$. Thus, $\mu(j) = f(\mu)(j) = j$. Now, we can calculate \begin{align*}
    G_\mu \cup U + ij &= G_\mu \cup U_{\id_{\Fer(G^\dagger \setminus G)}} + f(\mu)(i)f(\mu)(j) = (G \cup U + ij)_{f(\mu)}\\
    &= G^\dagger_{f(\mu)} = G^\dagger - ab + cd = (G \cup U + ij) - ab + cd\\
    &= (G - ab + cd) \cup U + ij.
\end{align*}
When we restrict our vertices to the set $V(G)$ on each side of this equation, we get $G_\mu = G - ab + cd$. Since $\mu(j) = j$, we have $\mu \in \mathcal{E}_{G} \subseteq S$. Thus, we have proved that for every case, $f(\mu) \in \mathcal{E}_{G^\dagger}^j$ implies that $\mu \in S$. Since we already proved the converse statement, we have $\sigma \in S $ if and only if $ f(\sigma) \in \mathcal{E}_{G^\dagger}^i$, and so we achieve the desired bijection, proving item i).

\textit{Proof of item ii).}
The bijection $f: \mathcal{E}_{G}^j \cup \Aut(G) \to \mathcal{E}_{G^\dagger}^i$ extends to a group isomorphism $\Sym(X \setminus \{i\}) \to \on{Stab}_{\Sym(X)}(i)$, so the subgroups generated by $\mathcal{E}_{G}^j \cup \Aut(G)$ and $\mathcal{E}_{G^\dagger}^i$ are isomorphic.

\textit{Proof of item iii).}
The graph $G^\dagger$ is stem-symmetric at $v_i$ if and only if $\Fer^i(G^\dagger) = \on{Stab}_{\Sym(X)}(i)$. This is true if and only if $\mathbf{H}_j(G) = \Sym(X \setminus \{i\})$, which is the case if and only if $G$ is hang-symmetric.
\end{proof}

\Cref{prop:hang-group} justifies the name \textit{hang group}, as the graph $G$ can be thought of as hanging from the vertex $v_i$ in $G^\dagger$. 

\section{Stem-/ Hang-Symmetry and Global Amoebas}~\label{sec: stem hang-sym global}
In this section, we provide characterizations when a graph $G^*$ is stem- and hang-symmetric  at a root. As a corollary, we obtain a novel characterization of global amoebas that supplements the list found in~\cite[Theorem 15]{caro2023graphs}.
\pagebreak

\begin{theorem}\label{thm:stem-sym-global}
    Let $G$ be a labeled graph rooted at a vertex labeled $i$ and let $X$ be the label set of $G^*$ with the new isolated vertex labeled $j$. Then the following are equivalent:
    \begin{enumerate}[label = (\alph*)]
        \item $G^*$ is stem-symmetric at $v_i$;
        \item $\Fer^i(G^*)$ acts transitively on $X \setminus \{i\}$.
        \item Every orbit of $\Fer^i(G)$ acting on $X \setminus \{i, j\}$ contains a leaf.
    \end{enumerate}
\end{theorem}

\begin{proof}
Immediately (a) implies (b) so assume (b) to prove (c). 

\noindent $(b)\implies (c)$. By assumption (b), for every $k \in X \setminus \{i, j\}$, there is a sequence of permutations in $\E_{G^*}^i$ whose composition takes $k$ to the label of an isolated vertex of $G^*$. Let $ \sigma_1, \dots, \sigma_n$ be such a sequence with minimum length and let $k_m := \sigma_m\sigma_{m-1}\cdots\sigma_1(k)$ for any $m\leq n$. Note that $k_n$ results in the label of an isolated vertex; by minimality, for $m < n$, $k_m = \sigma_m\sigma_{m-1}\cdots\sigma_1(k)$ is not an isolated vertex. Thus, for any $m < n$, we have that $k_m \in X \setminus \{i, j\}$ since $j$ labels an isolated vertex and $X \setminus \{i\}$ is stable under the action of $\Fer^i(G^*)$.

Since any member of $\E_{G^*}^i$ can only add $1$, $-1$, or $0$ to the degree of any vertex it moves, $k_{n-1}$ must be the label of a leaf since $n$ is minimal. Then $\sigma_1,\ldots,\sigma_{n-1}$ is a sequence of permutations in $\E_{G^*}^i$ that takes $k$ to the label of a leaf. However, we wish to find such a sequence of permutations $\tau_1, \dots, \tau_n \in \E_G^i$ 
that takes $k$ to the label of a leaf. To this end, we must find a way to replace $\sigma_m$'s that move $j$ with permutations that do not move $j$, and send $k_m$ to $k_{m+1}$ for $m < n-1$.\\

\noindent If $\sigma_m$ does not move $j$, then $\sigma_m\in \E^i_G$. Take $\tau_m := \sigma_m$ and note that $\tau_m(k_m)=\sigma_m(k_m)$. \\

\noindent On the other hand, say $\sigma_m$ moves $j$, and let $a = \sigma_m(j)$. 
If $a$ labels an isolated vertex, then $(j \; a)$ is an automorphism of $G^*$ fixing $i$ and $k_{m+1}$; note that $m+1 < n$ so  $k_{m+1}$ is not the label of an isolated vertex. Since $\E_{G^*}$ is closed under left-multiplication by $\Aut(G^*)$, $(j \; a) \sigma_m \in \E_{G^*}$. Moreover, $\sigma_m$ and $(j \; a)$ both fix $i$ so $(j \; a) \sigma_m \in \E_{G^*}^i$. Taking $\tau_m = (j \; a) \sigma_m$, observe that $\tau_m(j) = j$ so $\tau_m\in \E^i_G$; by construction, $\tau_m(k_m) = \sigma(k_m)$.  

If $a$ labels a leaf, the edge-replacement associated to $\sigma_m$ must create a new isolated vertex since $j$ is mapped by $\sigma_m$ to the label of a leaf. Therefore, this edge-replacement has the form $b\ell \to cj$, where $\ell$ is the label of a leaf. Let $G'$ be the graph $G$ with the leaf $\ell$ trimmed so that $G = G' \cup \{\ell\} + b\ell$. A direct calculation shows that
\begin{align*}
    G^* - b\ell + c\ell &= G \cup \{j\} - b\ell + c\ell\\
    &= G' \cup \{\ell\} \cup \{j\} + b\ell - b\ell + c\ell\\
    &= G' \cup \{\ell\} \cup \{j\} + c\ell\\
    &= (j \; \ell)(G' \cup \{j\} \cup \{\ell\} + cj)\\
    &= (j \; \ell)(G' \cup \{j\} \cup \{\ell\} + b\ell - b\ell + cj)\\
    &= (j \; \ell)(G \cup \{j\} - b\ell + cj)\\
    &= (j \; \ell)(G^* - b\ell + cj)\\
    &= (j \; \ell)\sigma_m(G^*)
\end{align*}

Therefore, $b\ell \to c\ell$ is a feasible edge-replacement in $G^*$ and consequently in $G$; since this edge-replacement does not increase the degree of $j$, $(j \; \ell)\sigma_m$ must take $j$ to some isolated vertex. If, $k_{m+1} = \ell$, $(j \; \ell)\sigma_m(k_m) = (j \; \ell)(k_{m+1}) = (j \; \ell) (\ell) = j$, and so $k_m$ is already the label of a leaf, and so $\tau_1, \dots, \tau_{m-1}$ is a sequence of permutations in $\E_G^i$ that takes $k$ to a label of a leaf, namely $k_m$. If $\ell \neq k_m$, then $(j \; \ell)\sigma_m(k_m) = (j \; \ell)k_{m+1} = k_{m+1}$, and so we may apply the logic in the case that $\sigma_m$ does \textit{not} increase the degree of $j$ to get a permutation $\tau_m$ that sends $k_m$ to $k_{m+1}$, fixes $j$, and fixes $i$.

\noindent By construction $\tau_m(k_m) = \sigma_m(k_m) = k_{m+1}$ and each $\tau_m$ is in $\E_G^i$ for all $m$. So $\tau_{n-1}\cdots \tau_1(k) = k_{n-1}$, which is a label of a leaf in the orbit of $k$ in $\Fer^i(G)$.\\ 

\noindent $(c)\implies (a)$. Finally assume (c) to prove (a), i.e., every orbit of $\Fer^i(G)$, except the singleton orbit $\{i\}$, contains a label of a leaf. It suffices to prove that the transposition $(k \; j)$ is in $\Fer^i(G^*)$ for any $k \in X \setminus \{i, j\}$. Suppose $k \in X \setminus \{i, j\}$. Then, there is some label $\ell$ of a leaf $v_{\ell}$ and some $\sigma \in \Fer^i(G) \leq \Fer^i(G^*)$ such that $\sigma(k) = \ell$. Let $m$ be the label of the unique neighbor of $v_{\ell}$. By performing the feasible edge-replacement $m\ell  \to mj$, we see that $(j \; \ell) \in \Fer^i(G^*)$. Thus, $(j \; k) = \sigma^{-1} (j \; \ell) \sigma \in \Fer^i(G)$ as desired. Since every transposition $(k \; j)$ lies in $\Fer^i(G^*)$ for $k \in X \setminus \{i, j\}$, then $\Sym(X \setminus \{i\}) \leq \Fer^i(G^*)$, and so $G^*$ is stem-symmetric at $v_i$.  
\end{proof}

From this theorem, we derive a useful result:

\begin{corollary}\label{cor:hang-sym-for-global-amoeba}
Let $G$ be a labeled graph rooted at label $i$.
The following are equivalent:
\begin{enumerate}[label = (\alph*)]
    \item $G^*$ is hang-symmetric at $i$;
    \item $\bold{H}_i(G^*)$ acts transitively on the label set of $G^*$.
    \item Every orbit of $\bold{H}_i(G)$, acting on the set of labels of $G$, contains a leaf with label not equal to $i$.
\end{enumerate}
\end{corollary}

\begin{proof}
Consider the graph $(G^\dagger)^*$. Note that $(G^\dagger)^* = (G^*)^\dagger$, and label the new leaf with the label $\ell$. Then by \Cref{prop:hang-group}, $\bold{H}_i(G) = \Fer^\ell(G^\dagger)$ and $\bold{H}_i(G^*) = \Fer^\ell(G^{*\dagger}) = \Fer^\ell(G^{\dagger *})$. Therefore, \Cref{thm:stem-sym-global} implies that $\bold{H}_i(G^*)$ is the full symmetric group if and only if it acts transitively on the label set of $G^*$. This occurs if and only if $\bold{H}_i(G)$ contains a label of a leaf of $G^\dagger$ in every orbit (except for the singleton $\{\ell\}$). Note that the labels of the leaves of $G^\dagger$ are exactly the labels of the leaves of $G$ which are not equal to $\ell$. This gives the result.
\end{proof}

From Theorem~\ref{thm:stem-sym-global}, we also obtain the following addition to the list of characterizations of global amoebas by Caro et al. In particular, note that a priori transitivity is a weaker condition than ~\cite[Theorem 15 (iii)]{caro2023graphs}.

\begin{corollary}~\label{cor: global amoeba = transitive}
Let $G$ be a graph. Then $G$ is a global amoeba if and only if $\Fer(G^*)$ acts transitively on its set of labels.
\end{corollary}
\begin{proof}
Consider the graph $G^{**}$ with label set $X$ and give the newly-added isolated vertex the label $j.$ Note that the feasible edge-replacements on $G^*$ act trivially on the label $j$, i.e., $\E_{G^*} \subseteq \E_{G^{**}}^j$. On the other hand, let $ab \to cd$ be a feasible edge-replacement associated to some $\sigma \in \E_{G^{**}}^j$. Since $j$ is isolated, then neither $a$ nor $b$ are equal to $j$. Moreover, neither $c$ nor $d$ can be equal to $j$, otherwise we move the label $j$. Thus, $ab \to cd$ is a feasible edge-replacement of $G^*$, and so $\sigma \in \E_{G^*}$. Thus $\E_{G^{**}}^j \subseteq\E_{G^*}$. 

Since the generating sets are equal, we have that $\Fer(G^*) = \Fer^j(G^{**})$. Now, \Cref{thm:stem-sym-global} states that the latter group is transitive if and only if it is the full symmetric group, and so this also holds for the former group. The result follows.
\end{proof}

\section{Comb Products of Amoebas}
\label{sec: Comb Product of Amoebas}

In this section, we study a way to combine local amoebas to achieve new local amoebas via the \emph{comb product}. This was studied in \cite[Theorem 3.10]{hansberg2021recursive}, where it was proved that the comb product $G * H$ of a nonempty global amoeba $G$ and a double-rooted global amoeba $H$ is, again, a global amoeba. We extend their result by proving that the comb product of two local amoebas is, under certain conditions, a local amoeba. 

In our study of local amoebas, it will be necessary to pursue more techniques from the field of permutation groups. See \nameref{appendix: wreath} for a brief discussion of block systems, primitive/imprimitive groups, and wreath products. A particularly relevant wreath product is $S_m \wr S_n$, which we let act on a set $B \times X$ using the imprimitive action where $|B| = m$ and $|X| = n$. We note the following fact about this group:

\begin{lemma}[\cite{DM96}]\label{Sm_wr_Sn_is_maximal}
Let $S_m$ act on a set $B$ and let $S_n$ act on a set $X$ both in the usual way, where $m, n \geq 2$. Then, the group $S_m \wr S_n$ is the largest subgroup of $S_{mn}$ which has $\{B \times \{x\} \mid x \in X\}$ as a block system, and moreover $S_m \wr S_n$ is itself a maximal subgroup of $S_{mn}$.
\end{lemma}

Our goal is to leverage the maximality of $S_m \wr S_n$ inside $S_{mn}$ to get a local amoeba construction using \Cref{Sm_wr_Sn_is_maximal}.  We now state the definition of comb product, which was originally used in \cite{hansberg2021recursive} in the setting of global amoebas. Let $G$ and $H$ be graphs, where $H$ has a root $v$. The \emph{comb product} $G * H$ is the graph with vertex set $V(G) \times V(H)$ constructed in the following way. For each edge $uu' \in E(G)$, place an edge between $(u, v)$ and $(u', v)$ in $G * H$.  For each $u \in V(G)$ and edge $ww' \in E(H)$, place an edge between $(u, w)$ and $(u, w')$ in $G * H$. In other words, to each vertex in $G$, glue a copy of $H$ by its root $v$. An example is provided in \Cref{fig:comb_product_example}. Given label sets $X$ and $B$ on graphs $G$ and $H$, respectively, we naturally obtain a label set $B\times X$ on the comb product $G\ast H$. Since our aim is to analyze the comb product of local amoebas, we begin by examining how this construction affects the feasible edge-replacement groups of the involved graphs. To that end, we now prove a technical lemma, followed by a sequence of corollaries that illustrate its implications in the context of local and global amoebas. See \Cref{fig:wreath_edge_replacements} for an example of the edge-replacements used in \Cref{lem:big_lemma}.

\begin{figure}[H]
    \centering

\begin{tikzpicture}
    \node (0) at (0, 0) {};
    \node (1) at (1.299,-0.75) {};
    \node (2) at (1.299,0.75) {};
    \node (3) at (-1.5,0) {};

    \draw (3) -- (0) -- (1) -- (2) -- (0);

    \draw[fill=white] (0) circle(0.25);
    \draw[fill=white] (1) circle(0.25);
    \draw[fill=white] (2) circle(0.25);
    \draw[fill=white] (3) circle(0.25);
\end{tikzpicture} \hspace{2cm}
\begin{tikzpicture}
\node (c0) at (0, 0) {};
\node (c1) at (1.5, 0) {};
\node (c2) at (3, 0) {};
\node (c3) at (0, -1.5) {};
\node (c4) at (1.5, -1.5) {};
\node (c5) at (3, -1.5) {};

\draw (c0) -- (c1) -- (c2);
\draw (c0) -- (c3);
\draw (c1) -- (c4);
\draw (c2) -- (c5);

\draw[fill=white!50!red] (c0) circle(0.25);
\draw[fill=white] (c1) circle(0.25);
\draw[fill=white] (c2) circle(0.25);
\draw[fill=white] (c3) circle(0.25);
\draw[fill=white] (c4) circle(0.25);
\draw[fill=white] (c5) circle(0.25);
\end{tikzpicture}

\begin{tikzpicture}

    \node (c0) at (0, 0) {};
    \node (c1) at (0, -1) {};
    \node (c2) at (0, -2) {};
    \node (c3) at (-1, -0.5) {};
    \node (c4) at (-1, -1.5) {};
    \node (c5) at (-1, -2.5) {};
    
    \draw (c0) -- (c1) -- (c2);
    \draw (c0) -- (c3);
    \draw (c1) -- (c4);
    \draw (c2) -- (c5);

    \node (d0) at (2, 0) {};
    \node (d1) at (2, -1) {};
    \node (d2) at (2, -2) {};
    \node (d3) at (1, -0.5) {};
    \node (d4) at (1, -1.5) {};
    \node (d5) at (1, -2.5) {};
    
    \draw (d0) -- (d1) -- (d2);
    \draw (d0) -- (d3);
    \draw (d1) -- (d4);
    \draw (d2) -- (d5);

    \node (e0) at (4, 0) {};
    \node (e1) at (4, -1) {};
    \node (e2) at (4, -2) {};
    \node (e3) at (3, -0.5) {};
    \node (e4) at (3, -1.5) {};
    \node (e5) at (3, -2.5) {};
    
    \draw (e0) -- (e1) -- (e2);
    \draw (e0) -- (e3);
    \draw (e1) -- (e4);
    \draw (e2) -- (e5);

    \node (f0) at (6, 0) {};
    \node (f1) at (6, -1) {};
    \node (f2) at (6, -2) {};
    \node (f3) at (5, -0.5) {};
    \node (f4) at (5, -1.5) {};
    \node (f5) at (5, -2.5) {};
    
    \draw (f0) -- (f1) -- (f2);
    \draw (f0) -- (f3);
    \draw (f1) -- (f4);
    \draw (f2) -- (f5);

    \draw (c0) -- (d0) -- (e0) -- (f0);
    \draw (c0) arc(135:45:2.828);

    \draw[fill=white!50!red] (c0) circle(0.25);
    \draw[fill=white] (c1) circle(0.25);
    \draw[fill=white] (c2) circle(0.25);
    \draw[fill=white] (c3) circle(0.25);
    \draw[fill=white] (c4) circle(0.25);
    \draw[fill=white] (c5) circle(0.25);
    
    \draw[fill=white!50!red] (d0) circle(0.25);
    \draw[fill=white] (d1) circle(0.25);
    \draw[fill=white] (d2) circle(0.25);
    \draw[fill=white] (d3) circle(0.25);
    \draw[fill=white] (d4) circle(0.25);
    \draw[fill=white] (d5) circle(0.25);
    
    \draw[fill=white!50!red] (e0) circle(0.25);
    \draw[fill=white] (e1) circle(0.25);
    \draw[fill=white] (e2) circle(0.25);
    \draw[fill=white] (e3) circle(0.25);
    \draw[fill=white] (e4) circle(0.25);
    \draw[fill=white] (e5) circle(0.25);
    
    \draw[fill=white!50!red] (f0) circle(0.25);
    \draw[fill=white] (f1) circle(0.25);
    \draw[fill=white] (f2) circle(0.25);
    \draw[fill=white] (f3) circle(0.25);
    \draw[fill=white] (f4) circle(0.25);
    \draw[fill=white] (f5) circle(0.25);

\end{tikzpicture}
    \caption{Pictured is the graph $G$ on the top left, the graph $H$ rooted on the red vertex on the top right, and their comb product $G * H$ on the bottom.}
    \label{fig:comb_product_example}
\end{figure}
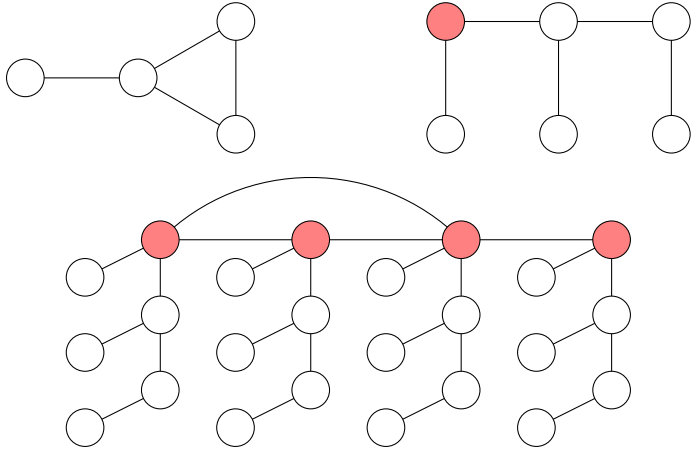

\begin{lemma}\label{lem:big_lemma}
If $G$ is a global amoeba with label set $X$ and $H$ is a rooted graph with label set $B$, where $i$ is the label of the root, then $\bold{H}_i(H) \wr \Fer(G)$ embeds into $\Fer(G * H)$. Moreover, the partition $\{B \times \{x\} \mid x \in X\}$ of $B \times X$, which is the label set of $G * H$, is a block system for $\bold{H}_i(H) \wr \Fer(G)$. 
\end{lemma}

\begin{proof}
Let $T = \Fer(G)$ and $S = \bold{H}_i(H)$. Define the map $\sigma \mapsto \widetilde{\sigma}$, where $\widetilde{\sigma}(b, x) = (b, \sigma(x))$, and let $\widetilde{T}$ be the image of $T$ under this map. Moreover, for any $x \in X$, let $\widetilde{S}_x\leq \Sym(B\times X)$ be the subgroup that acts on the block $B \times \{x\}$ by $\sigma(b, x) = (\sigma(b), x)$, for $\sigma \in S$, and stabilizes every other block. In order to embed the group $S \wr T$ into $\Fer(G * H)$, we wish to embed the groups $\widetilde{T}$, and $\widetilde{S}_x$ into $\Fer(G * H)$ for each $x \in X$. Let us start with $\widetilde{T}$. Translating \cite[Lemma 3.6]{hansberg2021recursive} to the language used in this paper, we can assert that $\E_G$ embeds into $\E_{G * H}$ by the map $\sigma \mapsto \widetilde{\sigma}$. In other words, if $xy \to wz$ is a feasible edge-replacement of $G$, then $(i, x)(i, y) \to (i, w)(i, z)$ is also a feasible edge-replacement of $G * H$. Moreover, if the former edge-replacement is associated to $\sigma$, then the latter is associated to $\widetilde{\sigma}$. Since $\E_G$ generates $T$, it follows that the image of $\E_G$ under this embedding generates $\widetilde{T}$, and $\widetilde{T} \leq \Fer(G * H)$ holds.

Next, we find $\widetilde{S}_y \leq \Fer(G * H)$, where $y$ is the label of a leaf in $G$, and $z$ is the label of its unique neighbor in $G$. Let $H_y$ be the subgraph of $G * H$ induced on the vertex set $B \times \{y\}$, and let $(i, y)$ be the root of $H_y$. Then $H_y$ is isomorphic to $H$ by definition. Since $y$ is a leaf in $G$, then $(i, y)(i, z)$ is the unique edge outgoing from $H_y$. Then, the graph induced on the vertex set $B \times \{y\} \cup \{(i, z)\}$ is equal to $H_y^\dagger$, where the newly added leaf is given the label $(i, z)$. Thus, $\bold{H}_y(H) \cong \bold{H}_{(i, y)}(H_y) = \Fer^{(i, z)}(H_y^\dagger) \leq \Fer(G * H)$, where the last inclusion follows from Lemma \ref{lemma:extends}. The image of this embedding is $\widetilde{S}_y$, and so we have $\widetilde{S}_y \leq \Fer(G * H)$.

Finally, we show that $\widetilde{S}_x \leq \Fer(G * H)$ for any $x \in X$. Let $x \in X$. By \cite[Theorem 15]{caro2023graphs}, there is a permutation $\sigma \in \Fer(G)$ and a label $y \in Y$ such that $y$ is the label of a leaf, and $\sigma(x) = y$. Since $\widetilde{S}_y \leq \Fer(G * H)$, and $\widetilde{\sigma} \in \Fer(G * H)$, then $\widetilde{S}_x = \widetilde{\sigma}^{-1}\widetilde{S}_y\widetilde{\sigma} \leq \Fer(G * H)$. Now, we have shown that $\widetilde{T} \leq \Fer(G * H)$ and $\widetilde{S}_x \leq \Fer(G * H)$ for every $x \in X$, and so we may conclude that $\bold{H}_i(H) \wr \Fer(G) \leq \Fer(G * H)$.
\end{proof}

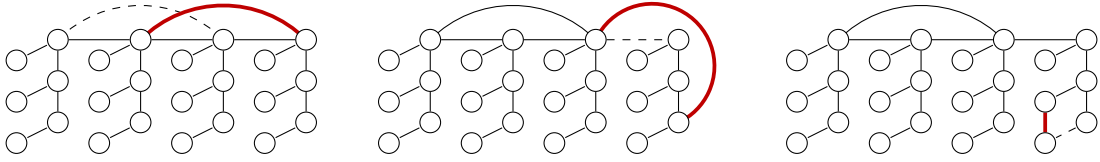
\begin{figure}[H]~\label{fig:wreath_edge_replacements}
    \centering
    $\;$\hfill
    \begin{tikzpicture}[scale=0.55]

\node (c0) at (0, 0) {};
\node (c1) at (0, -1) {};
\node (c2) at (0, -2) {};
\node (c3) at (-1, -0.5) {};
\node (c4) at (-1, -1.5) {};
\node (c5) at (-1, -2.5) {};

\draw (c0) -- (c1) -- (c2);
\draw (c0) -- (c3);
\draw (c1) -- (c4);
\draw (c2) -- (c5);

\node (d0) at (2, 0) {};
\node (d1) at (2, -1) {};
\node (d2) at (2, -2) {};
\node (d3) at (1, -0.5) {};
\node (d4) at (1, -1.5) {};
\node (d5) at (1, -2.5) {};

\draw (d0) -- (d1) -- (d2);
\draw (d0) -- (d3);
\draw (d1) -- (d4);
\draw (d2) -- (d5);

\node (e0) at (4, 0) {};
\node (e1) at (4, -1) {};
\node (e2) at (4, -2) {};
\node (e3) at (3, -0.5) {};
\node (e4) at (3, -1.5) {};
\node (e5) at (3, -2.5) {};

\draw (e0) -- (e1) -- (e2);
\draw (e0) -- (e3);
\draw (e1) -- (e4);
\draw (e2) -- (e5);

\node (f0) at (6, 0) {};
\node (f1) at (6, -1) {};
\node (f2) at (6, -2) {};
\node (f3) at (5, -0.5) {};
\node (f4) at (5, -1.5) {};
\node (f5) at (5, -2.5) {};

\draw (f0) -- (f1) -- (f2);
\draw (f0) -- (f3);
\draw (f1) -- (f4);
\draw (f2) -- (f5);

\draw (c0) -- (d0) -- (e0) -- (f0);
\draw[dashed] (c0) arc(135:45:2.828);
\draw[line width=0.5mm, color = red!75!black] (d0) arc(135:45:2.828);

\draw[fill=white] (c0) circle(0.25);
\draw[fill=white] (c1) circle(0.25);
\draw[fill=white] (c2) circle(0.25);
\draw[fill=white] (c3) circle(0.25);
\draw[fill=white] (c4) circle(0.25);
\draw[fill=white] (c5) circle(0.25);

\draw[fill=white] (d0) circle(0.25);
\draw[fill=white] (d1) circle(0.25);
\draw[fill=white] (d2) circle(0.25);
\draw[fill=white] (d3) circle(0.25);
\draw[fill=white] (d4) circle(0.25);
\draw[fill=white] (d5) circle(0.25);

\draw[fill=white] (e0) circle(0.25);
\draw[fill=white] (e1) circle(0.25);
\draw[fill=white] (e2) circle(0.25);
\draw[fill=white] (e3) circle(0.25);
\draw[fill=white] (e4) circle(0.25);
\draw[fill=white] (e5) circle(0.25);

\draw[fill=white] (f0) circle(0.25);
\draw[fill=white] (f1) circle(0.25);
\draw[fill=white] (f2) circle(0.25);
\draw[fill=white] (f3) circle(0.25);
\draw[fill=white] (f4) circle(0.25);
\draw[fill=white] (f5) circle(0.25);

\end{tikzpicture}
\hfill
\begin{tikzpicture}[scale=0.55]

\node (c0) at (0, 0) {};
\node (c1) at (0, -1) {};
\node (c2) at (0, -2) {};
\node (c3) at (-1, -0.5) {};
\node (c4) at (-1, -1.5) {};
\node (c5) at (-1, -2.5) {};

\draw (c0) -- (c1) -- (c2);
\draw (c0) -- (c3);
\draw (c1) -- (c4);
\draw (c2) -- (c5);

\node (d0) at (2, 0) {};
\node (d1) at (2, -1) {};
\node (d2) at (2, -2) {};
\node (d3) at (1, -0.5) {};
\node (d4) at (1, -1.5) {};
\node (d5) at (1, -2.5) {};

\draw (d0) -- (d1) -- (d2);
\draw (d0) -- (d3);
\draw (d1) -- (d4);
\draw (d2) -- (d5);

\node (e0) at (4, 0) {};
\node (e1) at (4, -1) {};
\node (e2) at (4, -2) {};
\node (e3) at (3, -0.5) {};
\node (e4) at (3, -1.5) {};
\node (e5) at (3, -2.5) {};

\draw (e0) -- (e1) -- (e2);
\draw (e0) -- (e3);
\draw (e1) -- (e4);
\draw (e2) -- (e5);

\node (f0) at (6, 0) {};
\node (f1) at (6, -1) {};
\node (f2) at (6, -2) {};
\node (f3) at (5, -0.5) {};
\node (f4) at (5, -1.5) {};
\node (f5) at (5, -2.5) {};

\draw (f0) -- (f1) -- (f2);
\draw (f0) -- (f3);
\draw (f1) -- (f4);
\draw (f2) -- (f5);

\draw (c0) -- (d0) -- (e0);
\draw[dashed] (e0) -- (f0);
\draw (c0) arc(135:45:2.828);
\draw[line width=0.5mm, color = red!75!black] (e0) arc(155:-65:1.505);

\draw[fill=white] (c0) circle(0.25);
\draw[fill=white] (c1) circle(0.25);
\draw[fill=white] (c2) circle(0.25);
\draw[fill=white] (c3) circle(0.25);
\draw[fill=white] (c4) circle(0.25);
\draw[fill=white] (c5) circle(0.25);

\draw[fill=white] (d0) circle(0.25);
\draw[fill=white] (d1) circle(0.25);
\draw[fill=white] (d2) circle(0.25);
\draw[fill=white] (d3) circle(0.25);
\draw[fill=white] (d4) circle(0.25);
\draw[fill=white] (d5) circle(0.25);

\draw[fill=white] (e0) circle(0.25);
\draw[fill=white] (e1) circle(0.25);
\draw[fill=white] (e2) circle(0.25);
\draw[fill=white] (e3) circle(0.25);
\draw[fill=white] (e4) circle(0.25);
\draw[fill=white] (e5) circle(0.25);

\draw[fill=white] (f0) circle(0.25);
\draw[fill=white] (f1) circle(0.25);
\draw[fill=white] (f2) circle(0.25);
\draw[fill=white] (f3) circle(0.25);
\draw[fill=white] (f4) circle(0.25);
\draw[fill=white] (f5) circle(0.25);

\end{tikzpicture}
\hfill
\begin{tikzpicture}[scale=0.55]

\node (c0) at (0, 0) {};
\node (c1) at (0, -1) {};
\node (c2) at (0, -2) {};
\node (c3) at (-1, -0.5) {};
\node (c4) at (-1, -1.5) {};
\node (c5) at (-1, -2.5) {};

\draw (c0) -- (c1) -- (c2);
\draw (c0) -- (c3);
\draw (c1) -- (c4);
\draw (c2) -- (c5);

\node (d0) at (2, 0) {};
\node (d1) at (2, -1) {};
\node (d2) at (2, -2) {};
\node (d3) at (1, -0.5) {};
\node (d4) at (1, -1.5) {};
\node (d5) at (1, -2.5) {};

\draw (d0) -- (d1) -- (d2);
\draw (d0) -- (d3);
\draw (d1) -- (d4);
\draw (d2) -- (d5);

\node (e0) at (4, 0) {};
\node (e1) at (4, -1) {};
\node (e2) at (4, -2) {};
\node (e3) at (3, -0.5) {};
\node (e4) at (3, -1.5) {};
\node (e5) at (3, -2.5) {};

\draw (e0) -- (e1) -- (e2);
\draw (e0) -- (e3);
\draw (e1) -- (e4);
\draw (e2) -- (e5);

\node (f0) at (6, 0) {};
\node (f1) at (6, -1) {};
\node (f2) at (6, -2) {};
\node (f3) at (5, -0.5) {};
\node (f4) at (5, -1.5) {};
\node (f5) at (5, -2.5) {};

\draw (f0) -- (f1) -- (f2);
\draw (f0) -- (f3);
\draw (f1) -- (f4);
\draw[dashed] (f2) -- (f5);
\draw[line width=0.5mm, color = red!75!black] (f4) -- (f5);

\draw (c0) -- (d0) -- (e0) -- (f0);
\draw (c0) arc(135:45:2.828);

\draw[fill=white] (c0) circle(0.25);
\draw[fill=white] (c1) circle(0.25);
\draw[fill=white] (c2) circle(0.25);
\draw[fill=white] (c3) circle(0.25);
\draw[fill=white] (c4) circle(0.25);
\draw[fill=white] (c5) circle(0.25);

\draw[fill=white] (d0) circle(0.25);
\draw[fill=white] (d1) circle(0.25);
\draw[fill=white] (d2) circle(0.25);
\draw[fill=white] (d3) circle(0.25);
\draw[fill=white] (d4) circle(0.25);
\draw[fill=white] (d5) circle(0.25);

\draw[fill=white] (e0) circle(0.25);
\draw[fill=white] (e1) circle(0.25);
\draw[fill=white] (e2) circle(0.25);
\draw[fill=white] (e3) circle(0.25);
\draw[fill=white] (e4) circle(0.25);
\draw[fill=white] (e5) circle(0.25);

\draw[fill=white] (f0) circle(0.25);
\draw[fill=white] (f1) circle(0.25);
\draw[fill=white] (f2) circle(0.25);
\draw[fill=white] (f3) circle(0.25);
\draw[fill=white] (f4) circle(0.25);
\draw[fill=white] (f5) circle(0.25);

\end{tikzpicture}
\hfill$\;$
    \caption{Let $G$ and $H$ be the graphs described in \Cref{fig:comb_product_example}. Using the notation in \Cref{lem:big_lemma}, the leftmost edge-replacement generates a permutation in $\widetilde{T}$, and the two rightmost generate permutations in $\widetilde{S}_x$ for some $x$.}~\label{fig:wreath_edge_replacements}
\end{figure}

In order to leverage the maximality of $S_m \wr S_n$ inside $S_{mn}$ to construct local amoebas, we require conditions that allow the embedding of $S_m \wr S_n$ into $\Fer(G * H)$.

\begin{corollary}\label{cor:big_corollary}
Let $G$ be a local amoeba with a leaf on $n$ vertices, and let $H$ be a rooted graph on $m$ vertices which is hang-symmetric at the root. Then the following are true:

\begin{enumerate}[label = (\arabic*)]
    \item $S_m \wr S_n \leq \Fer(G * H)$.
    \item $\Fer(G * H)$ is either equal to $S_m \wr S_n$ or $S_{mn}$.
    \item If there is some permutation in $\Fer(G * H)$ that does \textit{not} respect the block system $\{B \times \{x\} \mid x \in X\}$, then $G * H$ is a local amoeba.
\end{enumerate}
\end{corollary}

\begin{proof}
Since $G$ is a local amoeba with a leaf, it follows that $G$ is a global amoeba by \Cref{lem:local-am with leaf implies local and global}. Therefore, $(1)$ follows immediately from \Cref{lem:big_lemma}. Claim $(2)$ follows from the fact that $S_m \wr S_n$ is maximal in $S_{mn}$. To prove $(3)$, suppose that there is $\sigma \in \Fer(G * H)$ which does not respect the block system $\{B \times \{x\} \mid x \in X\}$. Since $S_m \wr S_n \leq \Fer(G * H)$ is the largest group that respects the block system, then $\sigma \notin S_m \wr S_n$. Thus, $\Fer(G * H)$ is strictly larger than $S_m \wr S_n$, and so it is equal to $S_{mn}$. Thus, $G * H$ is a local amoeba, and $(3)$ follows.
\end{proof}

Previously, stem-symmetric graphs with root-similar vertices have been studied in \cite{eslava2023new}. This condition is similar to hang-symmetry but strictly stronger. It turns out that requiring this condition of $H$ is also sufficient to construct a local amoeba as in \Cref{cor:big_corollary}.

\begin{corollary}\label{cor:also_big_corollary}
Let $G$ be a local amoeba on $n$ vertices with a leaf, and $H$ be a rooted graph on $m$ vertices with a root-similar vertex that is stem-symmetric at its root. Then $S_m \wr S_n \leq \Fer(G * H)$, and in particular, $\Fer(G * H)$ is either equal to $S_m \wr S_n$ or $S_{mn}$.
\end{corollary}
\begin{proof}
Let $i$ be the label of the root of $H$. Since $H$ is stem-symmetric at its root, then $S_{m-1} \leq \bold{H}_i(H)$, fixing the label $i$. Since there is an automorphism that moves the label $i$, and $S_{m-1}$ is maximal in $S_m$, then $\bold{H}_i(H) = S_m$, and so $H$ is hang-symmetric. Then the result follows by \Cref{cor:big_corollary}.
\end{proof}

If two graphs $G$ and $H$ satisfy the assumptions in \Cref{cor:big_corollary}, then to prove that $G * H$ is a local amoeba, it suffices to find a permutation in $\E_{G * H}$ or an automorphism that does not respect the block structure $\{B \times \{x\} \mid x \in X\}$ which is preserved by the wreath product $S_m \wr S_n$. Any such permutation is called a \emph{skew}. See the following example for a local amoeba constructed in this manner.

\begin{example}
Consider our running example of the graphs $G$ and $H$ from \Cref{fig:comb_product_example} where $v$ is any degree $2$ vertex of $H$. Call $X$ and $B$ the label sets of $G$ and $H$. We show that $G * H$ is a local amoeba. Consider the following labeling of the vertices of $H$ where $v$ is given the label $1$.
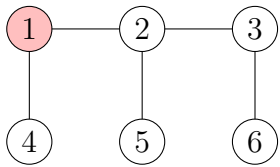
\begin{center}
\begin{tikzpicture}
\node (c0) at (0, 0) {};
\node (c1) at (1.5, 0) {};
\node (c2) at (3, 0) {};
\node (c3) at (0, -1.5) {};
\node (c4) at (1.5, -1.5) {};
\node (c5) at (3, -1.5) {};

\draw (c0) -- (c1) -- (c2);
\draw (c0) -- (c3);
\draw (c1) -- (c4);
\draw (c2) -- (c5);

\draw[fill=white!75!red] (c0) circle(0.3);
\draw[fill=white] (c1) circle(0.3);
\draw[fill=white] (c2) circle(0.3);
\draw[fill=white] (c3) circle(0.3);
\draw[fill=white] (c4) circle(0.3);
\draw[fill=white] (c5) circle(0.3);

\node at (c0) {$1$};
\node at (c1) {$2$};
\node at (c2) {$3$};
\node at (c3) {$4$};
\node at (c4) {$5$};
\node at (c5) {$6$};
\end{tikzpicture}
\label{fig:little-comb}
\end{center}

A quick calculation shows that $(2 \; 3)(5 \; 6),(5 \; 3),(3 \; 6)\in \mathcal{E}^1_H$ and $(1 \; 3)(4 \; 6)\in \Aut(H)$. Therefore $\mathbf{H}_1(H)=S_6$ as these four permutations generate $S_6$. Since $G$ is a local amoeba with a leaf, then it follows by Corollary \ref{cor:big_corollary} that $S_6 \wr S_4 \leq \Fer(G * H)$. 

From here, it remains to find a skew. Consider the following edge-replacement, where we remove the dashed edge and add the red edge.

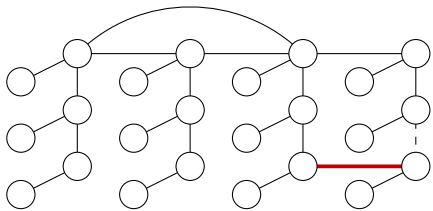
\begin{center}
\begin{tikzpicture}[scale=0.75]

\node (c0) at (0, 0) {};
\node (c1) at (0, -1) {};
\node (c2) at (0, -2) {};
\node (c3) at (-1, -0.5) {};
\node (c4) at (-1, -1.5) {};
\node (c5) at (-1, -2.5) {};

\draw (c0) -- (c1) -- (c2);
\draw (c0) -- (c3);
\draw (c1) -- (c4);
\draw (c2) -- (c5);

\node (d0) at (2, 0) {};
\node (d1) at (2, -1) {};
\node (d2) at (2, -2) {};
\node (d3) at (1, -0.5) {};
\node (d4) at (1, -1.5) {};
\node (d5) at (1, -2.5) {};

\draw (d0) -- (d1) -- (d2);
\draw (d0) -- (d3);
\draw (d1) -- (d4);
\draw (d2) -- (d5);

\node (e0) at (4, 0) {};
\node (e1) at (4, -1) {};
\node (e2) at (4, -2) {};
\node (e3) at (3, -0.5) {};
\node (e4) at (3, -1.5) {};
\node (e5) at (3, -2.5) {};

\draw (e0) -- (e1) -- (e2);
\draw (e0) -- (e3);
\draw (e1) -- (e4);
\draw (e2) -- (e5);

\node (f0) at (6, 0) {};
\node (f1) at (6, -1) {};
\node (f2) at (6, -2) {};
\node (f3) at (5, -0.5) {};
\node (f4) at (5, -1.5) {};
\node (f5) at (5, -2.5) {};

\draw (f0) -- (f1);
\draw[dashed] (f1) -- (f2);
\draw (f0) -- (f3);
\draw (f1) -- (f4);
\draw (f2) -- (f5);
\draw[line width=0.5mm, color = red!75!black] (f2) -- (e2);

\draw (c0) -- (d0) -- (e0) -- (f0);
\draw (c0) arc(135:45:2.828);

\draw[fill=white] (c0) circle(0.25);
\draw[fill=white] (c1) circle(0.25);
\draw[fill=white] (c2) circle(0.25);
\draw[fill=white] (c3) circle(0.25);
\draw[fill=white] (c4) circle(0.25);
\draw[fill=white] (c5) circle(0.25);

\draw[fill=white] (d0) circle(0.25);
\draw[fill=white] (d1) circle(0.25);
\draw[fill=white] (d2) circle(0.25);
\draw[fill=white] (d3) circle(0.25);
\draw[fill=white] (d4) circle(0.25);
\draw[fill=white] (d5) circle(0.25);

\draw[fill=white] (e0) circle(0.25);
\draw[fill=white] (e1) circle(0.25);
\draw[fill=white] (e2) circle(0.25);
\draw[fill=white] (e3) circle(0.25);
\draw[fill=white] (e4) circle(0.25);
\draw[fill=white] (e5) circle(0.25);

\draw[fill=white] (f0) circle(0.25);
\draw[fill=white] (f1) circle(0.25);
\draw[fill=white] (f2) circle(0.25);
\draw[fill=white] (f3) circle(0.25);
\draw[fill=white] (f4) circle(0.25);
\draw[fill=white] (f5) circle(0.25);
\end{tikzpicture}
\end{center}
This feasible edge-replacement generates a permutation which does not respect the block system $\{B\times \{x\}\mid x\in X\}$. Therefore, this permutation -- which is a skew -- lies outside $S_6 \wr S_4$. By maximality of $S_6 \wr S_4$ in $S_{24}$, it follows that $G * H$ is a local amoeba.
\end{example}

We discuss two cases of skews in the following lemma and theorem. A third case can be found in~\Cref{ex:not_paths}.

\begin{lemma}\label{lem:disconnected_comb_product}
Let $G$ and $H$ be local amoebas satisfying the assumptions of Corollary \ref{cor:big_corollary}. Moreover, assume that $H$ is disconnected. Then $G * H$ is a local amoeba.
\end{lemma}
\begin{proof}
If $H$ is disconnected, then $H$ decomposes into the disjoint union of graphs $H = H_1 \cup H_2$, and the label set $B$ of $H$ decomposes into the disjoint union $B = B_1 \cup B_2$, where $B_i$ is the label set of $H_i$ for $i \in \{1, 2\}$. Assume without loss of generality that the root of $H$ is in $V(H_1)$. Then $G * H$ is isomorphic to the disjoint union of $G * H_1$ and $|V(G)|$ copies of $H_2$. Moreover, each block $B \times \{x\}$ in the block system $\{B \times \{x\} \mid x \in X\}$ consists of the disjoint union of $B_1 \times \{x\}$ and $B_2 \times \{x\}$. Thus, for $x \neq y \in X$, there is an automorphism switching $B_2 \times \{x\}$ and $B_2 \times \{y\}$ (i.e., switching two of the $|V(G)|$ copies of $H_2$), and fixing everything else. This does not preserve the block system, otherwise we would also have to switch $B_1 \times \{x\}$ and $B_1 \times \{y\}$. Therefore, it follows by Corollary \ref{cor:big_corollary} that $G * H$ is a local amoeba.
\end{proof}

The following theorem adds to the work of Caro et al. in \cite{hansberg2021recursive} and provides an example of a skew. 

\begin{theorem}\label{thm:path-amoeba}
Let $P_n$ be the path on $n$ vertices rooted at a leaf. Then for any local amoeba $G$ with a leaf, we have that $G * P_n$ is a local amoeba.
\end{theorem}

\begin{proof}
Let $|V(G)| = k$. If $n = 1$ or $k = 1$, the proof is trivial, so for the rest of the proof, we assume $n > 1$ and $k > 1$. First, we wish to prove by induction that the path is hang-symmetric at any of its leaves. 

For $n = 2$, this is obvious as the automorphism group of $P_2$ is already the full symmetric group. Now suppose $n > 2$. Then $P_{n-1}^\dagger = P_n$ where $P_{n-1}$ is rooted at a leaf labelled $i$ and $P_{n-1}^\dagger = P_n$ is rooted at the newly-added leaf labelled $j$. By \Cref{prop:hang-group}, we have that $\Fer^j(P_n) = \Fer^j(P_{n-1}^\dagger) = \bold{H}_i(P_{n-1}) = S_{n-1}$, and so $P_n$ is stem-symmetric at $j$. However, as $j$ is the label of a leaf, there is an automorphism of $P_n$ that moves $j$, and so $\langle \E_{P_n}^j \cup \Aut(P_n) \rangle = \langle \Fer^j(P_n) \cup \Aut(P_n) \rangle = S_n$ together generate all of $S_n$. Thus, $P_n$ is hang-symmetric at a leaf for every $n \geq 2$.

Now, let $B = \{1, \dots, n\}$ be the label set of $P_n$, where $1$ is the label of a leaf, $2$ is the label of its unique neighbor, and so on. Moreover, let $n$ be the label of the root of $P_n$, and let $X$ be the label set of $G$. By \Cref{cor:big_corollary}, this means that $S_n \wr S_k \leq \Fer(G * P_n)$, the largest group preserving the block system $\{B \times \{x\} \mid x \in X\}$. To show that $G * P_n$ is a local amoeba, we simply need to find some skew, i.e., $\sigma \in \Fer(G * P_n)$ which does not preserve this block system.

Let $x \in X$ be the label of a leaf of $G$ and let $y \in X$ be the label of its unique neighbor. 

Then the following edge-replacement on $G * P_n$, where we replace the dashed edge with the red edge is feasible.

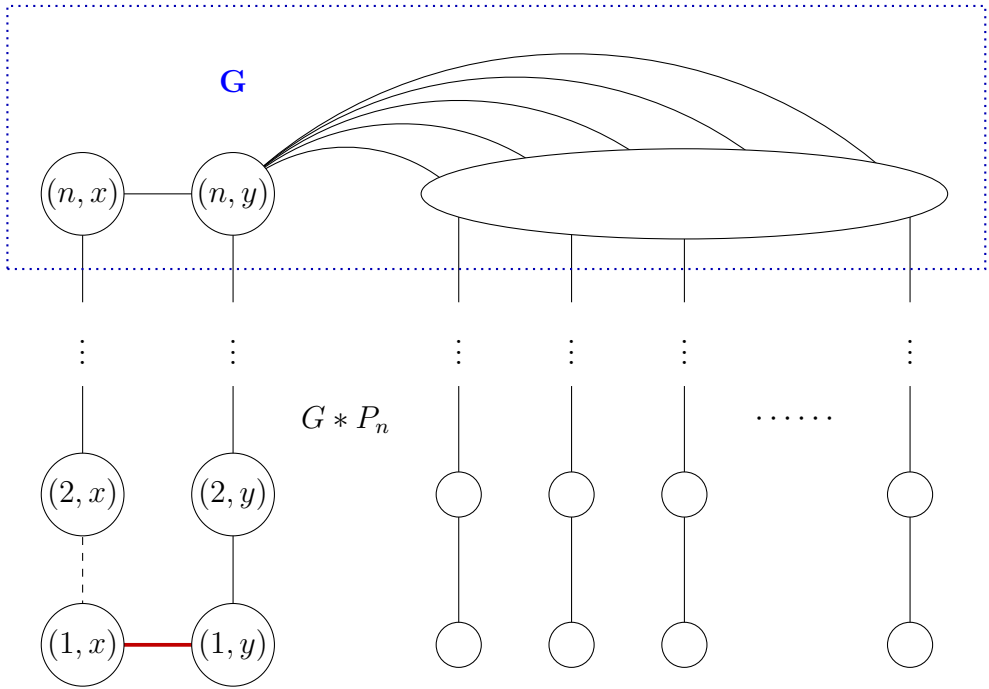
\begin{center}
\begin{tikzpicture}
\node (u1) at (0, 0) {};
\node (u2) at (0, -2) {};
\node (u3) at (0, -4) {};
\node (u4) at (0, -6) {};

\node (u'1) at (2, 0) {};
\node (u'2) at (2, -2) {};
\node (u'3) at (2, -4) {};
\node (u'4) at (2, -6) {};

\draw (u1) -- (u2) -- (u3);
\draw[dashed] (u3) -- (u4);
\draw (u'1) -- (u'2) -- (u'3) -- (u'4);

\draw[line width=0.5mm, color = red!75!black] (u4) -- (u'4);

\draw (5, 0) -- (5, -6);
\draw (6.5, 0) -- (6.5, -6);
\draw (8, 0) -- (8, -6);
\draw (11, 0) -- (11, -6);

\draw (u1) -- (u'1);
\draw (u'1) arc(135:45:2.121);
\draw (u'1) arc(135:45:3.182);
\draw (u'1) arc(135:45:4.243);
\draw (u'1) arc(135:45:5.303);
\draw (u'1) arc(135:45:6.364);

\draw[fill = white] (u1) circle(0.55);
\draw[white, fill = white] (u2) circle(0.55);
\draw[fill = white] (u3) circle(0.55);
\draw[fill = white] (u4) circle(0.55);

\draw[fill = white] (u'1) circle(0.55);
\draw[white, fill = white] (u'2) circle(0.55);
\draw[fill = white] (u'3) circle(0.55);
\draw[fill = white] (u'4) circle(0.55);

\draw[white, fill = white] (5, -2) circle(0.55);
\draw[white, fill = white] (6.5, -2) circle(0.55);
\draw[white, fill = white] (8, -2) circle(0.55);
\draw[white, fill = white] (11, -2) circle(0.55);

\draw[fill = white] (5, -4) circle(0.3);
\draw[fill = white] (6.5, -4) circle(0.3);
\draw[fill = white] (8, -4) circle(0.3);
\draw[fill = white] (11, -4) circle(0.3);

\draw[fill = white] (5, -6) circle(0.3);
\draw[fill = white] (6.5, -6) circle(0.3);
\draw[fill = white] (8, -6) circle(0.3);
\draw[fill = white] (11, -6) circle(0.3);

\draw[fill = white] (8, 0) ellipse (3.5 and 0.6);

\node at (u1) {$(n, x)$};
\node at (u2) {$\vdots$};
\node at (u3) {$(2, x)$};
\node at (u4) {$(1, x)$};

\node at (u'1) {$(n, y)$};
\node at (u'2) {$\vdots$};
\node at (u'3) {$(2, y)$};
\node at (u'4) {$(1, y)$};

\node at (5, -2) {$\vdots$};
\node at (6.5, -2) {$\vdots$};
\node at (8, -2) {$\vdots$};
\node at (11, -2) {$\vdots$};

\node at (9.5,-3) {$\cdots\cdots$};

\node at (2, 1.5) {\textcolor{blue}{$\textbf{G}$}};

\draw[color=blue!70!black, dotted, thick] (-1, -1) -- (12, -1) -- (12, 2.5) -- (-1, 2.5) -- cycle;

\node at (3.5, -3) {$G * P_n$};
\end{tikzpicture}
\end{center}

The edge-replacement depicted above, $(1, x)(2, x) \to (1, x)(1, y)$ is feasible. This edge-replacement induces the permutation $\sigma$ which swaps the labels $(m, y)$ and $(m+1, x)$ for each $1 \leq m \leq n-1$, and fixes all else. This permutation does not preserve the block system, and so it is a skew, as desired. By Corollary \ref{cor:big_corollary}, it follows that $G * P_n$ is a local amoeba.
\end{proof}

Since the comb product is associative, and for a local amoeba $G$ with a leaf, $G * P_n$ is also a local amoeba with a leaf, it follows that for any $n_1, n_2 \dots, n_k \in \mathbb{N}$ that $P_{n_1} * P_{n_2} * \cdots * P_{n_k}$ is a local amoeba, and $G * (P_{n_1} * \cdots * P_{n_k})$ is also a local amoeba. This construction is a strict generalization of the family $\mathcal{A}$ with the local amoeba property as proved in \cite{eslava2023new} since $A_n$ is the $n$-fold iterated comb product of $P_2$. We provide another example of two graphs $G$ and $H$ such that $H$ is \textit{not} the comb product of paths, and $G * H$ is a local amoeba in \Cref{ex:not_paths}.

\begin{theorem}\label{thm-ifH*hs-G*Hglobam}
Let $G$ be a global amoeba and, let $H$ be a rooted graph. If $H^*$ is hang-symmetric at its root, then $G * H$ is a global amoeba.
\end{theorem}
\begin{proof}
Let $i$ be the label of the root of $H^*$. Then by \Cref{cor:hang-sym-for-global-amoeba}, every orbit of $\bold{H}_i(H)$, acting on its label set, contains a leaf not labeled $i$. By \Cref{lem:big_lemma}, we may embed $\bold{H}_i(H) \wr \Fer(G)$ inside $\Fer(G * H)$. Therefore, if $\sigma \in \bold{H}_i(H)$, then the map $\widetilde{\sigma}: (b, x) \mapsto (\sigma(b), x)$ lives inside $\Fer(G * H)$. For any vertex $(b, x)\in V(G * H)$, we may find some $\sigma \in \bold{H}_i(H)$ such that $\sigma (b)=\ell\neq i$, where $\ell$ is the label of a leaf of $H$. This implies that $(\ell, x)$ is a leaf of $G * H$, and so $\widetilde{\sigma}(b, x) = (\sigma(b), x) = (\ell, x)$. Thus, for any label $k$ in $\Fer(G * H)$, there is some permutation $\varphi \in \Fer(G * H)$ that takes $k$ to the label of a leaf of $G * H$. Therefore, $G * H$ is a global amoeba.
\end{proof}

The following corollary states \Cref{thm-ifH*hs-G*Hglobam} in the stem-symmetric setting.

\begin{corollary}\label{cor:global_comb_product}
Let $G$ be a global amoeba, and let $H$ be a rooted graph. If $H^*$ is stem-symmetric at its root and has a root-similar vertex, then $G * H$ is a global amoeba.
\end{corollary}
\begin{proof}
Let $i$ be the label of the root of $H^*$. Because $H^*$ is stem-symmetric at its root, we know that $\Fer^i(H^*) = \Sym(B \setminus \{i\})$ where $B$ is the label set of $H^*$. Since $H^*$ has a root-similar vertex, there is an automorphism that moves the label $i$. Thus, $\bold{H}_i(H^*) = \langle \Aut(H^*) \cup \Sym(B \setminus \{i\}) \rangle  = \Sym(B)$. Thus, $H^*$ is hang-symmetric and the result follows by \Cref{thm-ifH*hs-G*Hglobam}.
\end{proof}

The following corollary states a sufficient condition for the existence of a skew.

\begin{corollary}\label{cor:G loc am and H=P_2*J then G*H loc am}
    Let $G$ be a local amoeba on $n$ vertices with a leaf, and let $H=P_2\ast J$ be stem symmetric at a vertex in $P_2$, where $J$ is some graph. Then $G\ast H$ is a local amoeba. 
\end{corollary}
\begin{proof}

Let $G$ be a local amoeba on $n$ vertices where $v_n$ is a leaf and $v_{n-1}$ is adjacent to $v_n$ in $G$. By \Cref{cor:also_big_corollary}, $\fer(G*H)$ is either equal to $S_m \wr S_n$ or $S_{mn}$. To prove that $G*H$ is a local amoeba, we use the structure of $H$ to find a skew in $\fer(G*H)$.  Note that $H=P_2\ast J$ admits a decomposition as $H=J \cup J+\{uu'\}$ where $u$ is the root and $u'$ is a root similar vertex.

Let $a$ be the label of $(v_n,u)$ in $G*H$, let $b$ be the label of $(v_n,u')$ in $G*H$. Finally, let $c$ be the label of $(v_{n-1},u)$ in the copy of $H$ hanging from $v_{n-1}$ in $G*H$. The feasible edge-replacement $ab\to cb$ induces the permutation $\rho$ which exchanges one copy of $J$ contained in the copy of $H$ hanging from $v_n$ and one copy of $J$ contained in the copy of $H$ hanging from $v_{n-1}$. Notice that $\rho$ breaks up the block system. Therefore $\fer(G*H)=S_{mn}$, which implies that $G*H$ is a local amoeba.
\end{proof}

A nontrivial example of a family of local amoebas that satisfy \Cref{cor:G loc am and H=P_2*J then G*H loc am} can be constructed if $J$ is a path rooted at a leaf. It is not difficult to see that such a graph provides sufficient structure for a skew in $\Fer(G*H)$. However, no other nontrivial examples of such graphs are known. This is stated as an open problem in \Cref{sec: Conclusion and open problems}.

\section{Stem-/Hang-Symmetry and Comb Products}
\label{sec: Stem-Symmetry Hang-Symmetry under the Comb product}

In this section, we study when stem-symmetry and hang-symmetry are preserved under the comb product. We prove that if $G^*$ and $H^*$ are stem- (resp. hang-) symmetric, then so is $(G * H)^*$.

\begin{lemma}\label{lem:fixed_fer_comb_product}
If $G$ and $H$ are rooted graphs whose roots are $i$ and $j$, respectively, then $\Fer^j(H) \wr \Fer^i(G) \leq \Fer^{(j, i)}(G * H).$ 
\end{lemma}
\begin{proof}
Let $S = \Fer^j(H)$ and $T = \Fer^i(G)$. Let $X$ and $B$ be the label sets of $G$ and $H$, respectively. Take $\widetilde{S}_x$ and $\widetilde{T}$ to be the same groups as denoted in proof of~\Cref{lem:big_lemma}. By \cite[Lemma 3.5]{hansberg2021recursive}, we may embed $\E_G^i$ into $\E_{G * H}$ by the map $\sigma \mapsto \widetilde{\sigma}$, where $\widetilde{\sigma}(b, x) = (b, \sigma(x))$. Since each member of the image of this embedding fixes $(j, i)$, then this is actually an embedding of $\E_G^i$ into $\E_{G * H}^{(j, i)}$. Therefore, this extends to an embedding of $T$ into $\Fer^{(j, i)}(G * H)$, whose image is $\widetilde{T}$. 

Next, pick $x \in X$. Then there is a copy of $H$ in $G * H$ whose vertices are $B \times \{x\}$, and the only outgoing edges of this copy of $H$ are connected to the vertex $(j, x)$. Therefore, we may embed $\E_H^j$ into $\E_{G * H}$ by the map $\tau \mapsto \widetilde{\tau}$ by $\widetilde{\tau}(b, x) = (\tau(b), x)$ and $\widetilde{\tau}(b, y) = (b, y)$ whenever $y \neq x$. Since each member of the image of this embedding fixes $(j, i)$, then this is actually an embedding of $\E_H^j$ into $\E_{G * H}^{(j, i)}$. Therefore, this embedding extends to an embedding of $S$ into $\Fer^{(j, i)}(G * H)$ whose image is $\widetilde{S}_x$. Since we have embedded $\widetilde{T}$ and each $\widetilde{S}_x$ into $\Fer^{(j, i)}(G * H)$, then we embed all of $\Fer^j(H) \wr \Fer^i(G)$ into $\Fer^{(j, i)}(G * H)$. 
\end{proof}
Note that this inequality is strict for many of the cases under study. If $G$ has a leaf $\ell$ with unique neighbor $\ell'$, and $H$ has a root $j$ and a root-similar vertex $k$, with $\sigma(j) = k$ for $\sigma \in \Aut(H)$, then the edge replacement $(j, \ell')(j, \ell) \to (j, \ell')(k, \ell)$ in $G * H$ moves the vertex $(j,\ell)$ to the vertex $(k,\ell)$. However, $\{j\} \times X$ is stable under the action of $\Fer^j(H) \wr \Fer^i(G)$, and so the inequality must be strict. However, often equality occurs, e.g.,  $\Fer^{(j,i)}(K_n\ast K_m)=\Fer^j(K_m)\wr \Fer^i(K_n)$ for $n,m\geq 3.$ 

\begin{theorem}\label{thm:when_is_comb_product_stem_sym_global}
If $G^*$ is stem-symmetric at its root labeled $i$ and $H^*$ is stem-symmetric at its root labeled $j$ with a root-similar vertex, then $(G * H)^*$ is stem-symmetric at the vertex labeled $(j, i)$. 
\end{theorem}

\begin{proof}
Let $B$ and $X$ be the label sets of $H$ and $G$ respectively. By \Cref{thm:stem-sym-global}, it suffices to show that every orbit of $\Fer^{(j, i)}(G * H)$ except the the singleton orbit $\{(j, i)\}$ contains a leaf. First, let $(b, x) \in B \times X$ such that $b \neq j$. Then there exists some $\sigma \in \Fer^j(H)$ such that $\sigma(b) = \ell$, where $\ell$ is a leaf of $H$ not equal to $j$. Then by \Cref{lem:fixed_fer_comb_product}, the permutation $\widetilde{\sigma}: (a, y) \mapsto (\sigma(a), y)$ on $B \times X$ lies in $\Fer^{(j, i)}(G * H)$. Thus $\widetilde{\sigma}(b, x) = (\sigma(b), x) = (\ell, x)$. Since $\ell$ is a leaf of $H$ not equal to $j$, then $(\ell, x)$ is a leaf of $G * H$. Thus, if $b \neq j$, then $(b, x)$ is in the same orbit as a leaf.

Now, suppose that $b = j$ and $x \neq i$ so that $(b, x) = (j, x)$.  Then there exists some $\tau \in \Fer^i(G)$ such that $\tau(x) = \ell$ a leaf, for $\ell \neq i$. By \Cref{lem:fixed_fer_comb_product}, we have $\widetilde{\tau} \in \Fer^{(j, i)}(G * H)$, where $\widetilde{\tau}: (b, y) \mapsto (b, \tau(y))$. So $\widetilde{\tau}(j, x) = (j, \ell)$. Now if $m$ is the unique neighbor of $\ell$ in $G$, and $k$ is a root similar vertex to $j$ in $H$, then the edge-replacement $(j, \ell)(j, m) \to (k, \ell)(j, m)$ is feasible. This moves the label $(j, \ell)$ to a label $(k', \ell)$ where $k' \neq j$; note that $k'$ may be in the orbit of $k$ under the action of $\Aut(H)$. So now we may apply the methods in the first paragraph to take this label to a leaf. Thus, $(j, x)$ is in the same orbit as a leaf, and we are done.
\end{proof}

This theorem justifies the study of graphs that are stem-symmetric with a root-similar vertex. If $H$ is merely hang-symmetric or merely stem-symmetric, then the above argument does not follow, and $G * H$ is not necessarily stem-symmetric, and indeed we can construct counterexamples in either case. We may also use this theorem to conclude that the property of ``becoming stem-symmetric with a root-similar vertex when an isolated vertex is added'' is closed under the comb product.

\begin{corollary}\label{cor:when_is_comb_product_stem_sym_root_sym_vertex_global}
If $G^*$ is stem-symmetric at its root $i$ with a root-similar vertex and $H^*$ is stem-symmetric at its root $j$ with a root-similar vertex, then $(G * H)^*$ is stem-symmetric at the label $(j, i)$, with a root-similar vertex.
\end{corollary}

\Cref{cor:when_is_comb_product_stem_sym_root_sym_vertex_global} is almost a direct consequence of \Cref{thm:when_is_comb_product_stem_sym_global} as the root-similar vertex property is inherited in $(G * H)^*$ thanks to the comb product.
This corollary also serves as a correction to \cite[Prop 3.11]{hansberg2021recursive}, replacing stem-symmetry of $G*H$ with stem-symmetry of $(G*H)^*$. The following is a counterexample to the latter statement which was confirmed by one of the authors via private communication.

\begin{counterexample}
Let $G$ be the path on $3$ vertices, rooted at a leaf, pictured below. The root is highlighted in red.

\begin{center}
\begin{tikzpicture}
\node (c0) at (0, 0) {};
\node (c1) at (1.5, 0) {};
\node (c2) at (3, 0) {};

\draw (c0) -- (c1) -- (c2);

\draw[fill=white!50!red] (c0) circle(0.25);
\draw[fill=white] (c1) circle(0.25);
\draw[fill=white] (c2) circle(0.25);
\end{tikzpicture}
\end{center}

By inspection, $G$ is stem-transitive, has a leaf, and has a root-similar vertex, and so it is a double-rooted global amoeba. Next, we let $H$ be the triangle with a pendant vertex, rooted at a degree $2$ vertex, pictured below with the root highlighted in red.

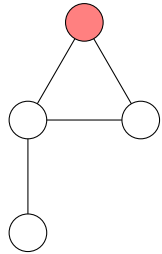
\begin{center}
\begin{tikzpicture}
    \node (0) at (-0.75, -1.299) {};
    \node (1) at (0.75, -1.299) {};
    \node (2) at (0, 0) {};
    \node (3) at (-0.75, -2.799) {};

    \draw (3) -- (0) -- (1) -- (2) -- (0);

    \draw[fill=white] (0) circle(0.25);
    \draw[fill=white] (1) circle(0.25);
    \draw[fill=white!50!red] (2) circle(0.25);
    \draw[fill=white] (3) circle(0.25);
\end{tikzpicture}
\end{center}

Again, it can be seen by inspection that $H$ is stem-transitive, has a root-similar vertex, and has a leaf, and so $H$ is a double-rooted global amoeba. Now let us consider $G * H$, pictured below with the root highlighted in red.

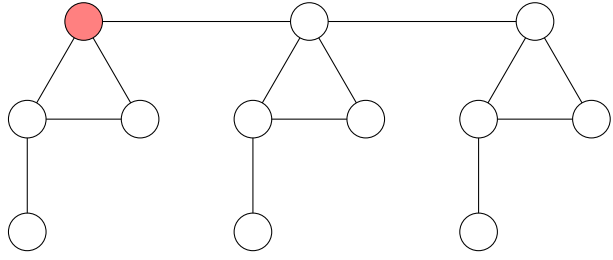
\begin{center}

\begin{tikzpicture}
    \node (a0) at (-0.75, -1.299) {};
    \node (a1) at (0.75, -1.299) {};
    \node (a2) at (0, 0) {};
    \node (a3) at (-0.75, -2.799) {};

    \draw (a3) -- (a0) -- (a1) -- (a2) -- (a0);

    \node (b0) at (2.25, -1.299) {};
    \node (b1) at (3.75, -1.299) {};
    \node (b2) at (3, 0) {};
    \node (b3) at (2.25, -2.799) {};

    \draw (b3) -- (b0) -- (b1) -- (b2) -- (b0);

    \node (c0) at (5.25, -1.299) {};
    \node (c1) at (6.75, -1.299) {};
    \node (c2) at (6, 0) {};
    \node (c3) at (5.25, -2.799) {};

    \draw (c3) -- (c0) -- (c1) -- (c2) -- (c0);

    \draw (a2) -- (b2) -- (c2);

    \draw[fill=white] (a0) circle(0.25);
    \draw[fill=white] (a1) circle(0.25);
    \draw[fill=white!50!red] (a2) circle(0.25);
    \draw[fill=white] (a3) circle(0.25);

    \draw[fill=white] (b0) circle(0.25);
    \draw[fill=white] (b1) circle(0.25);
    \draw[fill=white] (b2) circle(0.25);
    \draw[fill=white] (b3) circle(0.25);

    \draw[fill=white] (c0) circle(0.25);
    \draw[fill=white] (c1) circle(0.25);
    \draw[fill=white] (c2) circle(0.25);
    \draw[fill=white] (c3) circle(0.25);
\end{tikzpicture}
\end{center}

This graph has a leaf and a root-similar vertex, but it is not stem-transitive, so it is not a double-rooted global amoeba. Let us put down some labels on this graph:

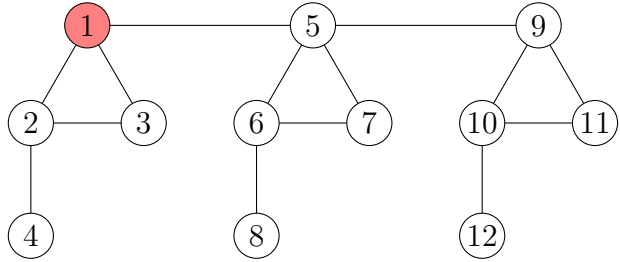
\begin{center}

\begin{tikzpicture}
    \node (a0) at (-0.75, -1.299) {};
    \node (a1) at (0.75, -1.299) {};
    \node (a2) at (0, 0) {};
    \node (a3) at (-0.75, -2.799) {};

    \draw (a3) -- (a0) -- (a1) -- (a2) -- (a0);

    \node (b0) at (2.25, -1.299) {};
    \node (b1) at (3.75, -1.299) {};
    \node (b2) at (3, 0) {};
    \node (b3) at (2.25, -2.799) {};

    \draw (b3) -- (b0) -- (b1) -- (b2) -- (b0);

    \node (c0) at (5.25, -1.299) {};
    \node (c1) at (6.75, -1.299) {};
    \node (c2) at (6, 0) {};
    \node (c3) at (5.25, -2.799) {};

    \draw (c3) -- (c0) -- (c1) -- (c2) -- (c0);

    \draw (a2) -- (b2) -- (c2);

    \draw[fill=white] (a0) circle(0.3);
    \draw[fill=white] (a1) circle(0.3);
    \draw[fill=white!50!red] (a2) circle(0.3);
    \draw[fill=white] (a3) circle(0.3);

    \draw[fill=white] (b0) circle(0.3);
    \draw[fill=white] (b1) circle(0.3);
    \draw[fill=white] (b2) circle(0.3);
    \draw[fill=white] (b3) circle(0.3);

    \draw[fill=white] (c0) circle(0.3);
    \draw[fill=white] (c1) circle(0.3);
    \draw[fill=white] (c2) circle(0.3);
    \draw[fill=white] (c3) circle(0.3);

    \node at (a2) {1};
    \node at (a0) {2};
    \node at (a1) {3};
    \node at (a3) {4};

    \node at (b2) {5};
    \node at (b0) {6};
    \node at (b1) {7};
    \node at (b3) {8};

    \node at (c2) {9};
    \node at (c0) {10};
    \node at (c1) {11};
    \node at (c3) {12};
\end{tikzpicture}
\end{center}

We claim that $\{2, 3, 4\}$ is an orbit of the group $\Gamma = \langle \E^1_{G * H}\rangle$. 

\begin{proposition}
Using the labeling on $G * H$ given above, the set $\{2, 3, 4\}$ is an orbit of the group $\Gamma = \langle \E^1_{G * H} \rangle$
\end{proposition}
\begin{proof}
Note that $G$ is a local amoeba with a leaf, and $H$ is stem-symmetric at its root with a root-similar vertex. Then, $\Fer(G * H)$ is either equal to $S_{12}$ or $S_4 \wr S_3$. Using the code given in \cite{marcos2024Github}, we found that $G * H$ is not a local amoeba, which implies that $\Fer(G * H)$ is equal to $S_4 \wr S_3$. Now, the stabilizer of $1$ in $\Fer(G * H) = S_4 \wr S_3$ is $\Sym(\{2, 3, 4\}) \times (S_4 \wr S_2)$, which is not transitive on $\{2, 3, \dots, 12\}$. In particular, it has two orbits, which are $\{2, 3, 4\}$ and $\{5, 6, \dots, 12\}$. Thus, the largest possible orbit of $\Gamma$ that contains the label $2$ is $\{2, 3, 4\}$. However, there are sequences of feasible edge-replacements taking $2$ to $3$ and $4$, so $\{2, 3, 4\}$ is an orbit of $\Gamma$. 
\end{proof}

Since $\Gamma$ does not act transitively, then $G * H$ is not stem-transitive, and so $G * H$ is not a double-rooted global amoeba. Thus, we have found graphs $G$ and $H$ such that $G$ and $H$ are double-rooted global amoebas, but $G * H$ is not. 

\end{counterexample}

Now let us focus our attention on hang-symmetry. The same analysis as in \Cref{lem:big_lemma} may be performed with the hang group to obtain the following result that is neither more nor less general than what is proved in \cite{hansberg2021recursive}.

\begin{lemma}\label{lem:big_lemma_but_for_hang_group}
If $G$ is a graph with root labeled $i$ such that $G^*$ is hang-symmetric, and $H$ is a graph with root labeled $j$, then $\bold{H}_j(H) \wr \bold{H}_i(G) \leq \bold{H}_{(j, i)}(G * H)$.
\end{lemma}
\begin{proof}
This proof follows similarly to \Cref{lem:big_lemma} except for a few key differences. By \Cref{lem:fixed_fer_comb_product} we have $\{1\} \wr \Fer^i(G) \leq \Fer^{(j, i)}(G * H)$. Note that
$\{1\} \wr \Aut(G) \leq \Aut(G * H)$.
Thus $\{1\} \wr \bold{H}_{i}(G) \leq \bold{H}_{(j, i)}(G * H)$. 

Now, we use the same techniques as \Cref{lem:big_lemma} to embed $\Aut(H)$ into $\E_{G * H}$ acting on each copy of $H$ inside $G * H$ corresponding to a leaf of $G$ whose label is not $i$. This fixes the label $(j, i)$, so we have actually embedded $\Aut(H)$ into $\E^{(j,i)}_{G * H}$, for each leaf of $G$ not equal to $i$. By \Cref{lem:fixed_fer_comb_product}, we already have $\Fer^j(H) \wr \{1\} \leq \Fer^{(j, i)}(G * H)$, and so there is a copy of $\Fer^j(H)$ acting on each copy of $H$. In particular, we have a copy of $\Fer^j(H)$ acting on each copy of $H$ corresponding to a leaf of $G$ whose label is not $i$. Therefore, we have a copy of $\bold{H}_j(H)$ acting on each such copy of $H$.

Now, since $G^*$ is hang-symmetric, then by \Cref{cor:hang-sym-for-global-amoeba}, every orbit of $\bold{H}_i(G)$ contains a leaf of $G$ whose label is not $i$. Since $\{1\} \wr \bold{H}_i(G) \leq \bold{H}_{(j, i)}(G * H)$, and we have a copy of $\bold{H}_j(H)$ acting on each copy of $H$ inside $G * H$ which corresponds to a leaf whose label is \textit{not} $i$, then by the same argument in \Cref{lem:big_lemma}, we have the full wreath product $\bold{H}_j(H) \wr \bold{H}_i(G)$ inside $\bold{H}_{(j, i)}(G * H)$.
\end{proof}

We would like for this inequality to be strict, but often it is not. Similar to~\Cref{lem:fixed_fer_comb_product}, the case of $P_n\ast K_m$ yields $\bold{H}_{(j, i)}(P_n * K_m)  = \bold{H}_j(K_m) \wr \bold{H}_i(P_n) = S_m \wr S_n < S_{mn}$, for $n,m\geq 3$. 

\begin{corollary}\label{cor:(G*H)* hang sym} If $G^*$ is hang-symmetric with root $i$, and $H^*$ is hang-symmetric with root labeled $j$, then $(G * H)^*$ is hang-symmetric with root labeled $(j, i)$.
\end{corollary}
\begin{proof}
By \Cref{cor:hang-sym-for-global-amoeba}, it suffices to show that every label $(b, x)$ of $G * H$ can be taken to a leaf of $G * H$. Since $H^*$ is hang-symmetric, then there exists some $\tau \in \bold{H}_j(H)$ such that $\tau(b) = \ell$, where $\ell$ is the label of a leaf not equal to $j$. Therefore, by \Cref{lem:big_lemma_but_for_hang_group}, there exists some $\widetilde{\tau} \in \bold{H}_{(j, i)}(G * H)$ such that $\widetilde{\tau}(a, y) = (\tau(a), y)$. Therefore, $\widetilde{\tau}(b, x) = (\tau(b), x) = (\ell, x)$. Since $\ell$ is the label of a leaf and $\ell \neq j$, then $(\ell, x)$ is the label of a leaf of $G * H$. Therefore, we may take every label $(b, x)$ of $G * H$ to the label of a leaf of $G * H$  via an element of $\bold{H}_{(j, i)}(G * H)$, and so $(G * H)^*$ is hang-symmetric.
\end{proof}
Using the fact that $S_m \wr S_n$ is a maximal subgroup of $S_{mn}$, we immediately deduce the following corollary.

\begin{corollary}\label{cor:hangsym under comb}
If $G$ has a leaf and is  hang-symmetric at a vertex labeled $i$ and $H$ is hang-symmetric at a vertex labeled $j$, then $\bold{H}_{(j, i)}(G * H)$ is either equal to $S_m \wr S_n$ or $S_{mn}$.
\end{corollary}

\begin{proof}
This follows similarly to the proof of \Cref{thm:path-amoeba}, provided we choose a leaf of $G$ whose label is not $i$. However, one is guaranteed to exist, because if a leaf of $G$ is labeled $i$, then there must exist $\sigma \in \Aut(G)$ that moves this label, otherwise $G$ is not hang-symmetric. Since $\Aut(G)$ acts by automorphisms, then $\sigma(i) \neq i$ must be the label of a leaf, which is not equal to our original leaf. 

Choosing a leaf of $G$ that does not have the label $i$ guarantees that the edge-replacement described in \Cref{thm:path-amoeba} fixes the label $i$, and so the permutation $\tau$ induced by this edge-replacement lies in $\E_{G* H}^{(j, i)} \leq \bold{H}_{(j, i)}(G * H)$. From here, the proof follows as in \Cref{thm:path-amoeba}.
\end{proof}
In particular, when $H$ is hang symmetric at every root in the case of a path $P_n$. 
\begin{corollary}\label{cor:hang-sym-path}
If $G$ has a leaf and is hang-symmetric at a vertex labeled $i$, and $P_n$ is the path on $n$ vertices, rooted at a leaf labeled $j$ then $G * P_n$ is hang-symmetric at the vertex labeled $(j, i)$. 
\end{corollary}

Using \Cref{cor:hang-sym-path}, we may further recover stem-symmetry of the family $\mathcal{B}$ shown in \cite{eslava2023new}. 
\begin{example}
In the prior reference, it was shown that $B_n$ is stem-symmetric at its highest-degree vertex, but we may show that it is stem-symmetric at a leaf connected to its highest degree vertex. First, we note that $B_n = (P_2 * P_2 * \cdots * P_2)^\dagger$, where we take the $n$-fold comb product and add a leaf. Since $P_2$ is hang-symmetric with a leaf, we may repeatedly apply \Cref{cor:hang-sym-path} to get that $P_2 * \cdots * P_2$ is hang-symmetric at a highest-degree vertex. It then follows by \Cref{prop:hang-group} that $B_n = (P_2 * \cdots * P_2)^\dagger$ is \textit{stem}-symmetric at the newly-added leaf. 
\end{example}
The major observation that allowed us to recover stem-symmetry of the family $\mathcal{B}$ is the decomposition of elements in the $\mathcal{A}$ family as iterated comb products of paths $P_2.$ 
As noted above, this decomposition can also be exploited to recover local amoeba property. Stem-symmetry at the maximal deegree vertex as shown in \cite{eslava2023new} can also be recovered via this decomposition, albeit with slightly more work.

We can also use this technique to find two graphs $G$ and $H$ such that $G * H$ is a local amoeba, and $H$ is not the comb product of paths.

\begin{example}\label{ex:not_paths}
We show that $P_2 * B_n$ is a local amoeba for $n \geq 2$. Define $A_n = P_2 * P_2 * \cdots * P_2$, then $n$-fold comb product of $P_2$. Then $A_n$ has a leaf connected to its root, and since $A_n^\dagger = B_n$ is rooted at the newly added leaf, then the root of $B_n$ is connected to a vertex, which is in turn connected to another leaf. Therefore, $B_n$ is double-rooted. Since we know that $B_n$ is stem-symmetric by the previous example, then $B_n$ is hang-symmetric. By \Cref{cor:big_corollary}, we have that $\Fer(P_2 * B_n)$ is either equal to $S_{2(2^n+1)}$ or $S_{2^n+1} \wr S_2$, so to show that $P_2 * B_n$ is a local amoeba, it suffices to find a skew.

Note that $B_n = (A_{n-2} * (P_2 * P_2))^\dagger$. Since $P_2 * P_2$ is $P_4$ rooted at a non-leaf, then there is a copy of $P_4$ hanging from a non-leaf at each vertex of $A_{n-2}$. When we apply the dagger operation, we add a leaf to the root of $A_{n-2}$, so there is a copy of $(P_2 * P_2)^\dagger$ hanging from the root of $A_{n-2}$. Since $B_n = (A_n)^\dagger$, then $B_n$ is rooted at the newly-added leaf. So omitting some detail, the structure of $P_2 * B_n$ can be said to look like this:

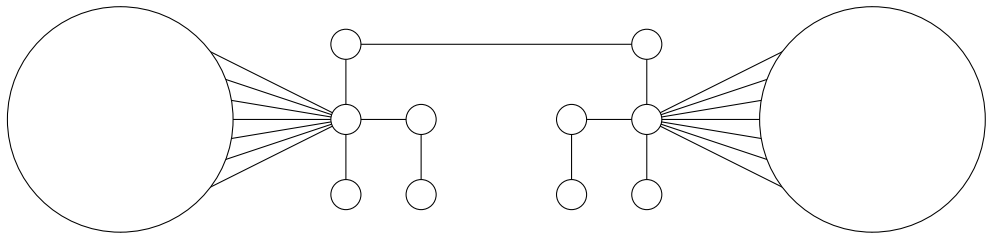
\begin{center}
\begin{tikzpicture}
    \node (1) at (0, 0) {};
    \node (2) at (0, -1) {};
    \node (3) at (0, -2) {};
    \node (4) at (1, -1) {};
    \node (5) at (1, -2) {};

    \draw (2) -- (-3, 0.5);
    \draw (2) -- (-3, 0);
    \draw (2) -- (-3, -0.5);
    \draw (2) -- (-3, -1);
    \draw (2) -- (-3, -1.5);
    \draw (2) -- (-3, -2);
    \draw (2) -- (-3, -2.5);

    \node (1') at (4, 0) {};
    \node (2') at (4, -1) {};
    \node (3') at (4, -2) {};
    \node (4') at (3, -1) {};
    \node (5') at (3, -2) {};

    \draw (2') -- (7, 0.5);
    \draw (2') -- (7, 0);
    \draw (2') -- (7, -0.5);
    \draw (2') -- (7, -1);
    \draw (2') -- (7, -1.5);
    \draw (2') -- (7, -2);
    \draw (2') -- (7, -2.5);

    \draw (1) -- (2) -- (4) -- (5);
    \draw (2) -- (3);

    \draw (1') -- (2') -- (4') -- (5');
    \draw (2') -- (3');

    \draw (1) -- (1');

    \draw[fill=white] (1) circle(0.2);
    \draw[fill=white] (2) circle(0.2);
    \draw[fill=white] (3) circle(0.2);
    \draw[fill=white] (4) circle(0.2);
    \draw[fill=white] (5) circle(0.2);

    \draw[fill=white] (1') circle(0.2);
    \draw[fill=white] (2') circle(0.2);
    \draw[fill=white] (3') circle(0.2);
    \draw[fill=white] (4') circle(0.2);
    \draw[fill=white] (5') circle(0.2);

    \draw[fill=white] (-3, -1) circle(1.5);

    \draw[fill=white] (7, -1) circle(1.5);
\end{tikzpicture}
\end{center}

The two blocks under the action of $S_{2^n+1} \wr S_2$ are the ``left'' and ''right'' side of this picture. To show that $P_2 * B_n$ is a local amoeba, we need to find some edge replacement that mixes these two blocks. However, we may perform the following edge replacement, where we remove the dashed edge and replace it with the red edge:

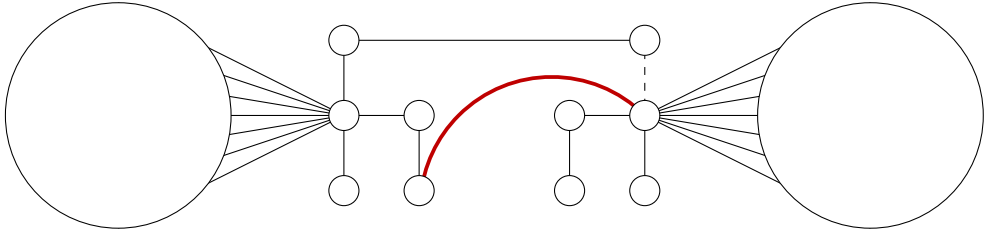
\begin{center}
\begin{tikzpicture}
    \node (1) at (0, 0) {};
    \node (2) at (0, -1) {};
    \node (3) at (0, -2) {};
    \node (4) at (1, -1) {};
    \node (5) at (1, -2) {};

    \draw (2) -- (-3, 0.5);
    \draw (2) -- (-3, 0);
    \draw (2) -- (-3, -0.5);
    \draw (2) -- (-3, -1);
    \draw (2) -- (-3, -1.5);
    \draw (2) -- (-3, -2);
    \draw (2) -- (-3, -2.5);

    \node (1') at (4, 0) {};
    \node (2') at (4, -1) {};
    \node (3') at (4, -2) {};
    \node (4') at (3, -1) {};
    \node (5') at (3, -2) {};

    \draw (2') -- (7, 0.5);
    \draw (2') -- (7, 0);
    \draw (2') -- (7, -0.5);
    \draw (2') -- (7, -1);
    \draw (2') -- (7, -1.5);
    \draw (2') -- (7, -2);
    \draw (2') -- (7, -2.5);

    \draw (1) -- (2) -- (4) -- (5);
    \draw (2) -- (3);

    \draw (2') -- (4') -- (5');
    \draw (2') -- (3');
    \draw[dashed] (1') -- (2');

    \draw (1) -- (1');

    \draw[line width=0.5mm, color = red!75!black] (2') arc(45:175:1.75);

    \draw[fill=white] (1) circle(0.2);
    \draw[fill=white] (2) circle(0.2);
    \draw[fill=white] (3) circle(0.2);
    \draw[fill=white] (4) circle(0.2);
    \draw[fill=white] (5) circle(0.2);

    \draw[fill=white] (1') circle(0.2);
    \draw[fill=white] (2') circle(0.2);
    \draw[fill=white] (3') circle(0.2);
    \draw[fill=white] (4') circle(0.2);
    \draw[fill=white] (5') circle(0.2);

    \draw[fill=white] (-3, -1) circle(1.5);

    \draw[fill=white] (7, -1) circle(1.5);
\end{tikzpicture}
\end{center}
This is a feasible edge-replacement that mixes the two blocks of the $S_{2^n+1} \wr S_2$ action. Therefore, $P_2 * B_n$ is a local amoeba.
\end{example}

\section{Conclusion and open problems}
\label{sec: Conclusion and open problems}

    Our results provide a foundation for systematic exploration of stem- and hang-symmetry, interpolation, and group actions in the context of labeled graph families.  By introducing the hang group invariant, we capture how local amoebas embed into larger ones and provided a mechanism for encoding the permutations generated by feasible edge-replacements. Our analysis of stem- and hang-symmetry establishes a series of necessary and sufficient conditions that characterize when a rooted graph becomes a local or global amoeba. These conditions, together with the operations of adding leaves and isolated vertices, give constructive methods for building new amoeba families. Furthermore, our use of the wreath product to define the comb product of amoebas strengthens and generalizes previously known constructions. Beyond these contributions, many questions remain open; some examples are included below. This indicates that the study of amoeba graphs remains a rich and fertile area of research.
\begin{itemize}

    \item \emph{If $H=P_2 * J$ is a graph that is stem-symmetric at a vertex in $P_2$ and $J$ is a nonempty graph, are there nontrivial examples where $J$ is not a path rooted at a leaf?}
    
    \item \emph{Which additional graph operations (beyond adding leaves or isolated vertices) preserve the amoeba property?} 
    
    \item \emph{Given a rooted local amoeba, what graph-theoretic conditions are sufficient to guarantee hang-symmetry at the root?}
    
    \item \emph{Which groups are realizable as $\Fer(G)$ groups for some graph $G$?} Note that~\Cref{cor: global amoeba = transitive} tells us that if $G$ contains an isolated vertex, the only realizable transitive $\Fer(G)$ group is the symmetric group.
    
    \item \emph{Is it possible to get amoeba constructions via different maximal subgroups of the symmetric group?} In particular, we would be interested in amoeba constructions that utilize primitive groups. Note that these are classified by the O'Nan--Scott Theorem. Known constructions of local/global amoebas so far leverage the following embeddings. The wreath product $S_m\wr S_n\hookrightarrow S_{mn}$ was exploited in~\cite{hansberg2021recursive} for the global amoeba property of comb product; in this paper, the construction is further generalized to address the local amoeba property along with  hang- and stem-symmetry. The  $S_n\hookrightarrow S_{n+1}$ embedding is used to construct local amoebas from stem-symmetric graphs. The $S_n\times S_m\hookrightarrow S_{n+m}$ embedding is used in~\cite[Lemma 5, Theorem 8]{eslava2023new}. The local amoeba property of Fibonacci tree and paths are excellent examples of this.   
    
    \item \emph{What can be said about the orbits of $\Fer$ groups of global amoebas?}  Global amoebas, in general, have $\Fer$ groups that are not transitive. Are there any arithmetic restrictions that arise? Moreover, for a graph $G$, the action of $\Aut(G)$ on its vertices tends to induce orbits that are a union of disconnected vertices, whereas the action of $\Fer(G)$ can induce orbits that have larger connected components. Are there any topological restrictions on these orbits? 
\end{itemize}

\subsection*{Acknowledgments}
Many of the computations supporting the results were done through useful code provided in~\cite{laffitte2023detection, marcos2024Github} and \cite{GAP4}. All authors would like to thank the Algebra and Discrete Mathematics seminar at University of California Davis where this collaboration began. The fourth author is grateful for the financial support received from the UC Davis Chancellor’s Postdoctoral Fellowship Program, United States of America.
\pagebreak 

\section*{Appendix}~\label{appendix: wreath}
\vspace{-20pt}
\subsection*{Wreath Product}
We mostly cite Chapters 1.5 and 2.6 in \cite{DM96} for the information contained in this section.

Let $\Gamma$ be a group acting transitively on a set $X$. Moreover, let $\Delta \subseteq X$ such that for any $\sigma \in \Gamma$, we have either $\sigma(\Delta) = \Delta$ or $\sigma(\Delta) \cap \Delta = \varnothing$. Then $\Delta$ is called a \emph{block}. The collection of sets $\{\sigma(\Delta) \mid \sigma \in \Gamma\}$, which partitions $X$, is called a \emph{block system}.

\begin{example}
Let $X$ be the vertices of the cube graph, and let $\Gamma$ be its automorphism group. Any set of two vertices with no common neighbors form a block. Geometrically, this is because any two such points are ``antipodal'' on the cube, and an automorphism of the cube must take each set of antipodal points to another set of antipodal points. Therefore, the collection of all pairs of vertices that are distance three apart forms a block system. See \Cref{fig:block_system_cube_example}.
\end{example}

\begin{figure}[h]
    \centering
    \begin{tikzpicture}[scale=0.75]
\node at (0, 0, 0) (000) {};
\node at (3, 0, 0) (100) {};
\node at (0, 3, 0) (010) {};
\node at (3, 3, 0) (110) {};
\node at (0, 0, 3) (001) {};
\node at (3, 0, 3) (101) {};
\node at (0, 3, 3) (011) {};
\node at (3, 3, 3) (111) {};

\draw (001) -- (011);
\draw (001) -- (101);
\draw (001) -- (000);

\draw (010) -- (110);
\draw (010) -- (011);
\draw (010) -- (000);

\draw (100) -- (101);
\draw (100) -- (110);
\draw (100) -- (000);

\draw[line width=2mm, color=white] (111) -- (110);
\draw[line width=2mm, color=white] (111) -- (101);
\draw[line width=2mm, color=white] (111) -- (011);

\draw (111) -- (110);
\draw (111) -- (101);
\draw (111) -- (011);

\draw[fill=white!50!red] (000) circle(0.25);
\draw[fill=white!50!blue] (100) circle(0.25);
\draw[fill=white!50!green] (010) circle(0.25);
\draw[fill=white!50!orange] (110) circle(0.25);
\draw[fill=white!50!orange] (001) circle(0.25);
\draw[fill=white!50!green] (101) circle(0.25);
\draw[fill=white!50!blue] (011) circle(0.25);
\draw[fill=white!50!red] (111) circle(0.25);

\end{tikzpicture}
    \caption{Any two vertices that have the same color form a block. The collection of pairs of same-colored vertices forms a block system. Any automorphism of the cube must move a same-colored pair to a same-colored pair.}
    \label{fig:block_system_cube_example}
\end{figure}
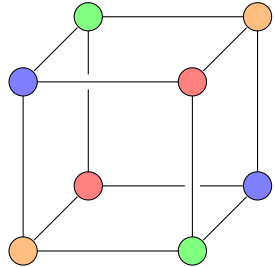

Note that if a group $\Gamma$ acts transitively on a set $X$, then any singleton $\{x\} \subseteq X$ is always a block, and the entire subset $X$ is always a block. However, we may find examples of permutation groups for which these are the only two blocks. Such groups are called \emph{primitive}. Likewise, a permutation group which has a nontrivial block is called \emph{imprimitive}. Of particular interest to our study of amoebas is the fact that the symmetric group $S_n$ itself is primitive, which distinguishes it from any group with a nontrivial block system.

Let us turn our attention now to imprimitive groups, and in particular, a construction of a type of group that often acts imprimitively. Let $B$ and $X$ be (finite) sets, and let $S \leq \Sym(B)$ and $T \leq \Sym(X)$ be permutation groups. Note that for every $x_0 \in X$, the group $S$ acts on $B \times X$ by $\sigma(b, x_0) = (\sigma(b), x_0)$, and $\sigma(b, x) = (b, x)$ when $x \neq x_0$, for $\sigma \in S$. Denote the image of this group action in $\Sym(B \times X)$ by $\widetilde{S}_{x_0}$. So $\widetilde{S}_{x_0}$ acts nontrivially on the block $B \times \{x_0\}$, and trivially on $B \times \{x\}$, for $x \neq x_0$. Moreover, the group $T$ acts on $B \times X$ by $\tau(b, x) = (b, \tau(x))$ for any $\tau \in T$. Denote the image of this group action in $\Sym(B \times X)$ by $\widetilde{T}$. 

Using the terminology from the previous paragraph, we define the \emph{wreath product} of $S$ and $T$ as the subgroup of $S \wr T \leq \Sym(B \times X)$ generated by $\widetilde{T}$ and $\widetilde{S}_x$ for every $x \in X$.

If $\tau(x) = y$ for some $\tau \in T$, then for its image $\widetilde{\tau} \in \widetilde{T}$, we have that $\widetilde{S}_y = \widetilde{\tau}^{-1}\widetilde{S}_x\widetilde{\tau}$. This means that $\widetilde{T}$ takes the subgroups $\widetilde{S}_x$ to each other by conjugation. Moreover, each $\widetilde{S}_x$ preserves the block system $\{B \times \{x\} \mid x \in X\}$, as does $\widetilde{T}$, so it follows that the wreath product $S \wr T$ does as well.

\bibliographystyle{plain}
\bibliography{bibliography}

\begin{thebibliography}{1}

\bibitem{caro2021unavoidable}
Yair Caro, Adriana Hansberg, and Amanda Montejano.
\newblock Unavoidable chromatic patterns in 2-colorings of the complete graph.
\newblock {\em Journal of Graph Theory}, 97(1):123--147, 2021.

\bibitem{caro2023graphs}
Yair Caro, Adriana Hansberg, and Amanda Montejano.
\newblock Graphs isomorphisms under edge-replacements and the family of amoebas.
\newblock {\em The Electronic Journal of Combinatorics}, pages P3--9, 2023.

\bibitem{DM96}
John~D. Dixon and Brian Mortimer.
\newblock {\em Permutation Groups}, volume 163 of {\em Graduate Texts in Mathematics}.
\newblock Springer-Verlag, 1996.

\bibitem{eslava2023new}
Laura Eslava, Adriana Hansberg, Tonatiuh Matos-Wiederhold, and Denae Ventura.
\newblock New recursive constructions of amoebas and their balancing number.
\newblock {\em Aequationes mathematicae}, pages 1--35, 2025.

\bibitem{GAP4}
The GAP~Group.
\newblock {\em {GAP -- Groups, Algorithms, and Programming, Version 4.15.0}}, 2025.

\bibitem{hansberg2021recursive}
Adriana Hansberg, Amanda Montejano, and Yair Caro.
\newblock Recursive constructions of amoebas.
\newblock {\em Procedia Computer Science}, 195:257--265, 2021.

\bibitem{laffitte2023detection}
Marcos E.~González Laffitte, J.~René González-Martínez, and Amanda Montejano.
\newblock On the detection of local and global amoebas: theoretical insights and practical algorithms (brief announcement).
\newblock {\em Procedia Computer Science}, 223:376--378, 2023.
\newblock XII Latin-American Algorithms, Graphs and Optimization Symposium (LAGOS 2023).

\bibitem{marcos2024Github}
Marcos E.~González Laffitte and Amanda Montejano.
\newblock Amoebas.
\newblock \url{https://github.com/MarcosLaffitte/Amoebas}, 2024.

\end{thebibliography}

\end{document}